\documentclass[12pt,leqno,amsfonts,amscd]{amsart}
\usepackage{epsfig}
\usepackage{amsmath}
\usepackage{amsfonts}
\usepackage{amssymb}
\usepackage{color}
\usepackage{hyperref}

\newtheorem{theorem}{Theorem}[section]
\newtheorem{lemma}[theorem]{Lemma}
\newtheorem{prop}[theorem]{Proposition}
\newtheorem{cor}[theorem]{Corollary}

\theoremstyle{definition}
\newtheorem{definition}[theorem]{Definition}
\newtheorem{example}[theorem]{Example}

\theoremstyle{remark}

\newtheorem{remark}[theorem]{Remark}

\numberwithin{equation}{section}

\newcommand{\NN}{{\mathbb N}}

\newcommand{\RR}{{\mathbb R}}

\newcommand{\eps}{\varepsilon}
\newcommand{\out}[1]{\ }
 
\DeclareMathOperator{\Ext}{Ext} 
\DeclareMathOperator{\fine}{fine}

\let\cal=\mathcal
\renewcommand{\phi}{\varphi}

\begin{document}
\title[Martin boundary of a fine domain]{Martin boundary of a
fine domain and a Fatou-Na{\"\i}m-Doob theorem for finely superharmonic
functions}

\author{Mohamed El Kadiri}
\address{Universit\'e Mohammed V
\\D\'epartement de Math\'ematiques
\\Facult\'e des Sciences
\\B.P. 1014, Rabat
\\Morocco}
\email{elkadiri@fsr.ac.ma}

\author{Bent Fuglede}
\address{Department of Mathematical Sciences
\\Universitetsparken 5
\\2100 Copenhagen
\\Denmark}
\email{fuglede@math.ku.dk}

\begin{abstract}
We construct the Martin compactification $\overline U$ of a fine
domain $U$ in $\RR^n$ ($n\ge 2$) and the Riesz-Martin kernel $K$ on
$U\times\overline U$. We obtain the integral representation of finely
superharmonic fonctions $\ge 0$ on $U$ in terms of $K$ and
establish the Fatou-Naim-Doob theorem in this setting.

\end{abstract}
\maketitle

\section{Introduction}\label{sec1}

The fine topology on an open set $\Omega\subset\RR^n$ was introduced
by H.\ Cartan in classical potential theory. It is defined as the
smallest topology on $\Omega$ making every superharmonic function on
$\Omega$ continuous. This topology is neither locally compact nor metrizable.
The fine topology has, however, other good
properties which allowed the development in the 1970's of a `fine'
potential theory on a finely open set $U\subset\Omega$, starting with the book
\cite{F1} of the second named author. The harmonic and superharmonic functions
and the potentials in this
theory are termed finely [super]harmonic functions and fine
potentials. Generally one distinguishes by the prefix `fine(ly)'
notions in fine potential theory from those in classical potential
theory on a usual (Euclidean) open set. Very many results from classical
potential theory have been extended to fine potential theory. In this
article we study the invariant functions, generalizing the
non-negative harmonic functions in the classical Riesz decomposition
theorem; and the integral representation of finely superharmonic
functions in terms of the `fine' Riesz-Martin kernel. The Choquet
representation theorem plays a key role. We close by
establishing the Fatou-Na{\"\i}m-Doob theorem on the fine
limit of finely superharmonic functions at the fine Martin boundary,
inspired in particular by the axiomatic approach of Taylor \cite{T}.

In forthcoming continuations \cite{EF2} and \cite{EF3} of the present paper
we study sweeping on a subset of the Riesz-Martin space, and the Dirichlet
problem at the Martin boundary of $U$.

Speaking in slightly more detail we consider the standard $H$-cone
$\cal S(U)$ of all finely superharmonic functions $\ge0$ on a given fine
domain $U$ (that is, a finely connected finely open subset of
a Greenian domain $\Omega\subset\RR^n$).  Generalizing
the classical Riesz representation theorem it was shown in
\cite{F4}, \cite{F5} that every function $u\in\cal S(U)$ has a unique
integral representation of the form
$$u(x)=\int G_U(x,y)d\mu(y)+h(x), \quad x\in \Omega,$$
where $\mu$ is a (positive) Borel measure on $U$,
$G_U$ is the (fine) Green kernel for $U$, and
$h$ is an invariant function on $U$. The term `invariant' reflects a property
established in \cite[Theorem 4.4]{F4} and generalizing fine harmonicity.

An interesting problem is whether every minimal invariant function is
finite valued and hence finely harmonic on $U$,
or equivalently whether every invariant function is the sum of a
sequence of finely harmonic functions. A negative answer to this question was
recently obtained (in dimension $n>2$) by Gardiner and Hansen \cite{GH}.

\centerline{*****}

Recall from \cite[Theorem 8.1]{F1} that a function
$u:U\longmapsto\,]-\infty,+\infty]$ is said to be
\textit{finely hyperharmonic} if (i) $u$ is finely l.s.c.\ and (ii) the
induced fine topology on $U$ has a base consisting of finely open sets $V$
of fine closure $\widetilde V\subset U$ such that
\begin{eqnarray}
u(x)\ge\int^*u\,d\eps_x^{\complement V}
\end{eqnarray}
for every $x\in V$ (complements being taken relative to $\Omega$).
Recall that the swept measure $\eps_x^{\complement V}$ (serving as a fine harmonic
measure) is carried by the fine boundary $\partial_fV\subset U$ and does not
charge the polar sets. As shown by Lyons \cite{L}, condition (ii) can be
replaced equivalently by the requirement that every point $x\in U$ has a fine
neighborhood base consisting of finely open sets $V$ with
$\widetilde V\subset U$ such that (1.1) holds. Every finely hyperharmonic
function is \textit{finely continuous} (\cite[Theorem 9.10]{F1}).

A finely hyperharmonic function $u$ on $U$ is said to be \textit{finely
superharmonic} if $u$ is not identically $+\infty$ on any fine component of
$U$, or equivalently if $u$ is finite on a finely dense subset of $U$, or
still equivalently if $u<+\infty$ quasi-everywhere (q.e.) on $U$.
We denote by $\cal S(U)$ the convex cone of all \textit{positive} finely
superharmonic functions on $U$, which is henceforth supposed to be finely
connected (a fine domain).

In \cite{El1} the first named author has defined a topology on the cone
${\cal S}(U)$, generalizing the topology of R.-M.\ Herv\'e in classical
potential theory. By identifying this
topology with the natural topology, now on the standard
$H$-cone ${\cal S}(U)$, it was shown in \cite{El1} that
${S}(U)$ has a compact base $B$, and
by Choquet's theorem that every function $u\in\cal S(U)$ admits a unique
integral representation of the form
$$u(x)=\int_Bs(x)d\mu(s), \quad x\in U,$$
where $\mu$ is a finite mesure on $B$ carried by the
(Borel) set of all extreme finely superharmonic functions belonging to $B$.

In the present article we define the Martin compactification $\overline U$ and
the Martin boundary $\Delta(U)$ of $U$.
While the Martin boundary of a usual open set is closed and hence compact, all
we can say in the present setup is that $\Delta(U)$ is a
$G_\delta$ subset of the compact Riesz-Martin space
$\overline U=U\cup\Delta U$ endowed with the natural topology. Nevertheless
we can define a Riesz-Martin kernel
$K:U\times\overline U\longrightarrow\,]0,+\infty]$ with good properties of
lower semicontinuity and measurability.
Every function $u\in\cal S(U)$ has an integral representation
$u(x)=\int_{\overline U}K(x,Y)d\mu(Y)$ in terms of a Radon measure $\mu$ on
$\overline U$. This representation is unique if it is required that $\mu$ be
carried by $U\cup\Delta_1(U)$ where $\Delta_1(U)$ denotes the minimal Martin
boundary of $U$, which is likewise a $G_\delta$ in $\overline U$. In that case
we write $\mu=\mu_u$. It is shown that, for any Radon measure $\mu$ on
$\overline U$, the associated function $u=\int K(.,Y)d\mu(Y)\in\cal S(U)$ is a
fine potential, resp.\ an invariant function, if and only if $\mu$ is carried
by $U$, resp.\ $\Delta(U)$.

There is a notion of minimal thinness of a set $E\subset U$ at a point
$Y\in\Delta_1(U)$, and an associated minimal-fine filter $\cal F(Y)$.
As a generalization of the classical Fatou-Na{\"\i}m-Doob theorem we show that
for any finely superharmonic function $u\ge0$ on $U$
and for $\mu_1$-almost every point $Y\in\Delta_1(U)$,
$u(x)$ has the limit $(d\mu_u/d\mu_1)(Y)$ as $x\to Y$ along
the minimal-fine filter $\cal F(Y)$. Here $d\mu_u/d\mu_1$ denotes the
Radon-Nikod{\'y}m derivative of the absolutely continuous component of
$\mu_u$ with respect to the absolutely continuous component of the measure
$\mu_1$ representing the constant function 1, which is finely harmonic and
hence invariant. Actually, we establish for any given invariant function $h>0$
the more general $h$-relative version of this result.\\

{\bf Notations}: For a Greenian domain $\Omega\subset\RR^n$ we denote by
$G_\Omega$ the Green kernel for
$\Omega$.  If $U$ is a fine domain in $\Omega$ we denote by
${\cal S}(U)$ the convex cone of finely superharmonic functions $\ge 0$
on $U$ in the sense of \cite{F1}. The convex cone of fine potentials
on $U$ (that is, the functions in ${\cal S}(U)$ for which every finely
subharmonic minorant is $\le 0$) is denoted by ${\cal P}(U)$. The cone of
invariant functions on $U$ is denoted by ${\cal H_i}(U)$; it is the
orthogonal band to ${\cal P}(U)$ relative to ${\cal S}(U)$. By $G_U$
we denote the (fine) Green kernel for $U$, cf.\ \cite{F2}. If $A\subset U$
and $f:A\longrightarrow[0,+\infty]$ one denotes by $R{}_f^A$, resp.\
${\widehat R}{}_f^A$, the reduced function, resp.\ the swept function, of $f$
on $A$ relative to $U$, cf.\ \cite[Section 11]{F1}.
For any set $A\subset\Omega$ we denote by $\widetilde A$ the
fine closure of $A$ in $\Omega$, and by $b(A)$ the base of $A$ in $\Omega$,
that is, the $G_\delta$ set of points of $\Omega$ at which $A$ is not thin,
in other words the set of all fine limit points of $A$ in $\Omega$.
We define $r(A)=\complement b(\complement A)$ (complements and bases relative
to $\Omega$). Thus $r(A)$ is a $K_\sigma$ set,
the least regular finely open subset of $\Omega$ containing the fine interior
$A'$ of $A$, and $r(A)\setminus A$ is polar.

\medskip\noindent
\textbf{Acknowledgment.} The Authors thank a referee for valuable references to
work related to Theorem 2.8 and Corollary 2.9, as described in Remark 2.10.

\section{The natural topology on the cone ${\cal S}(U)$}
\label{sec2}

We begin by establishing some basic properties of invariant functions
on an arbitrary fine domain $U\subset\Omega$ ($\Omega$ a Greenian domain in
$\RR^n$, $n\ge2$). First an auxiliary lemma of a general nature:

\begin{lemma}\label{lemma2.0} Every finely continuous function
$f:U\to\overline\RR$ is Borel measurable in the relative Euclidean topology
on $U\subset\Omega$.
\end{lemma}

\begin{proof} We shall reduce this to the known particular case where
$U=\Omega$, see \cite{F1a}. Since $\overline\RR$ is homeomorphic with $[0,1]$
we may assume that $f(U)\subset[0,1]$. Recall that the base $b(\phi)$ of
a function $\phi:\Omega\to[0,1]$ is defined as the finely derived function
$b(\phi):\Omega\to[0,1]$:
$$
b(\phi)(x)=\underset{y\to x,\,y\in\Omega\setminus{x}}{\fine\lim\sup}\,\phi(y)
$$
for $x\in\Omega$, see \cite[p.\ 590]{Do2}, \cite{F1a}.
Extend the given finely continuous function $f:U\to[0,1]$ to functions
$f_0,f_1:\Omega\to[0,1]$ by defining $f_0=0$, $f_1=1$ on $\Omega\setminus U$.
Then $f_1$ and $1-f_0$ are finely u.s.c.\ on $\Omega$, and hence
$b(f_1)\le f_1$ and $b(1-f_0)\le1-f_0$. Since $f_0\le f_1$ it follows that
$f_0\le b(f_0)\le b(f_1)\le f_1$. But $f_0=f_1=f$ on $U$, and so
$f=b(f_0)=b(f_1)$ on $U$. According to \cite{F1a}, $b(f_1)$ (and
$b(1-f_0)$) is of class $\cal G_\delta(\Omega)$, that is,
representable as the pointwise infimum of a (decreasing) sequence of
Euclidean l.s.c.\ functions $\Omega\to[0,+\infty]$. Consequently, $f$
(and $1-f$) is of class $\cal G(U)$, that is representable as the pointwise
infimum of a sequence of l.s.c.\ functions $U\to[0,+\infty]$, $U$ being given
the relative Euclidean topology. In particular, $f$ is Borel measurable in
that topology on $U$.
\end{proof}

\begin{lemma}\label{lemma2.1} If \,$h$ is invariant on $U$,
if \,$u\in{\cal S}(U)$, and if \,$h\le u$, then $h\preccurlyeq u$.
\end{lemma}

\begin{proof} There is a polar set $E\subset U$ such
that $h$ is finely harmonic on $U\setminus E$, and hence
$u-h$ is finely superharmonic on $U\setminus E$.  According to
\cite[Theorem 9.14]{F1},
$u-h$ extends by fine continuity to a function $s\in{\cal
S}(U)$ such that $h+s=u$ on $U\setminus E$ and hence on all of $U$,
whence $h\preccurlyeq u$.
\end{proof}

\begin{lemma}\label{lemma2.2} If \,$u\in{\cal S}(U)$ and $A\subset U$ then
$R_u^A(x)=\int_U u\,d\eps_x^{A\cup(\Omega\setminus U)}$ for $x\in U$.
\end{lemma}

\begin{proof} The integral exists because $u\ge0$ is Borel measurable in the
relative Euclidean topology on $U\subset\Omega$ by Lemma \ref{lemma2.0}.
Since $\widehat R{}_u^A$ and $\int_Uu\,d\eps_x^{A\cup(\Omega\setminus U)}$
remain unchanged if $A$ is replaced by its base $b(A)\cap U$ relative to $U$
and since the base operation is idempotent
we may assume that $A=b(A)\cap U$. Consequently, ${\widehat R}{}_u^A=R_u ^A$
and the set $V:=U\setminus A$ is
finely open (and regular). Let $u_0$ denote the
extension of $u$ to $\widetilde U=U\cup\partial_fU$ by the value $0$ on
$\partial_fU$.
Every function $s\in\cal S(U)$ with $s\ge u$ on $A$ is an upper function
(superfunction) for $u_0$ relative to $V$, see \cite[\S\S14.3--14.6]{F1}
concerning the (generalized) fine Dirichlet problem. It follows that
 $$
{\widehat R}{}_u^A(x)=R_u^A(x)\ge\overline H{}_{u_0}^V(x)
=\int^*_\Omega u_0\,d\eps_x^{\Omega\setminus V}
=\int^*_U u\,d\eps_x^{A\cup(\Omega\setminus U)}$$
for $x\in V$. For the opposite inequality, consider any upper function
$v$ for $u_0$ relative to $V$. In particular, $v\ge-p$ on $V$ for some finite
and hence semibounded potential $p$ on $\Omega$.  Define
$w=u\wedge v$ on $V$ and $w=u$ on $A=U\setminus V$. Then
 $$\underset{y\to x,\,y\in V}{\fine\lim\inf}\,w(y)
\ge u(x)\wedge\underset{y\to x,\,y\in V}{\fine\lim\inf}\,v(y)=u(x)$$
for every $x\in U\cap\partial_fV$. It therefore follows by
\cite[Lemma 10.1]{F1} that $w$ is finely hyperharmonic on $U$. Moreover,
$w$ is an upper function for $u$ relative to $V$ because
$w=u\wedge v\ge 0\wedge(-p)=-p$ on $V$. Since $w=u$ on $A$ we have
$w\ge {\widehat R}{}_u^A$ on $U$, and in particular
$v\ge w\ge {\widehat R}{}_u^A$ on $V$. By varying $v$ we obtain
$\overline H{}_{u_0}^V\ge{\widehat R}{}_u^A\ge R_u^A$ on $V$. Altogether we have
established the asserted equality for $x\in V$. It also holds for
$x\in U\setminus V=A=b(A)\cap U$ because
$\widehat R{}_u^A(x)=R_u^A(x)=u(x)$ and because
$\eps_x^{A\cup(\Omega\setminus U)}=\eps_x$, noting that
$x\in b(A)\subset b(A\cup(\Omega\setminus U))$.
\end{proof}

\begin{lemma}\label{lemma2.3} Let $u\in\cal S(U)$ and let $A$ be a subset of
$U$. The restriction of ${\widehat R}{}_u^A$ to any finely open subset $V$ of
\,$U\setminus A$ is invariant.
\end{lemma}

\begin{proof} We have ${\widehat R}{}_u^A
=\sup_{k\in\NN}{\widehat R}{}_{u\wedge k}^A$. Each of the functions
${\widehat R}{}_{u\wedge k}^A$ is finely harmonic on $V$
by \cite[Corollary 11.13]{F1}, and hence ${\widehat R}{}_u^A$ is invariant on
$V$ in view of Lemma \ref{lemma2.1} because the invariant
functions on $V$ form a band (the orthogonal band to $\cal P(V)$).
\end{proof}

As shown by the following example, the invariant functions on finely open
subsets of $U$ do not form a sheaf.

\begin{example}\label{example2.4} Let $\mu$ be the one-dimensional Lebesgue
measure on a line segment $E$ in $\Omega:=U:=\RR^3$. Then $E$ is polar and
hence everywhere thin in $\RR^3$. Every point $x\in\RR^3$ therefore has a
fine neighborhood $V_x$ with $\mu(V_x)=0$. (For  $x\in E$ take
$V_x=\{x\}\cup\complement L$, where $L$ denotes the whole line extending $E$.)
Thus $\mu$ does not have a (minimal) fine
support (unlike measures which do not charge any polar set). The Green
potential $G_\Omega\mu$ is invariant on each $V_x$ (but of course not on
$\Omega$). For if $p$ denotes a non-zero fine potential on $V_x$, $x\in E$,
such that $p\preccurlyeq G_\Omega\mu$ on $V_x$ then $p$ is finely harmonic on
$V_x\setminus\{x\}=\complement L$ along with $G_\Omega\mu$. Hence $p$ behaves
on $V_x$ near $x$ like a constant times $G_{V_x}\eps_x$ and is therefore of the
order of magnitude $1/r$, where $r$ denotes the distance from $x$, cf.\
\cite[Th\'eor\`eme]{F2}. But on the plane through $x$ orthogonal to $L$,
$G_\Omega\mu$ behaves like a
constant times $\log(1/r)$, in contradiction with $p\le G_\Omega\mu$.
\end{example}

The invariant functions on finely open subsets of $U$ do, however, have
a kind of countable sheaf property, as shown by (a) and (b) in the following
theorem:

\begin{theorem}\label{thm2.5} {\rm{(a)}} Let $u\in\cal S(U)$ be invariant
and let $V$ be a finely open subset of \,$U$. Then $u_{|V}$ is
invariant.

{\rm{(b)}} Let $u\in\cal S(U)$, and let $(U_j)$ be a countable cover
of \,$U$ by regular finely open subsets of $U$. If each $u_{|U_j}$ is invariant
then $u$ is invariant.

{\rm{(c)}} Let $(u_\alpha)$ be a decreasing net of invariant functions in
$\cal S(U)$. Then $\widehat\inf_\alpha\,u_\alpha$  is invariant. Moreover, the set
$V=\{\inf_\alpha u_\alpha<+\infty\}$ is finely open, and
$\widehat\inf_\alpha\,u_\alpha=\inf_\alpha u_\alpha$ on $V$.
\end{theorem}

\begin{proof} (a) 
Let $p\in\cal P(V)$ satisfy $p\preccurlyeq u$ on $V$. Write
$p=G_V\mu$ (considered on $V$) in terms of the associated Borel
measure $\mu$ on $\Omega$ such that the inner measure $\mu_*(\complement V)=0$.
We shall prove that $p=0$. An extension of $\mu$ from $V$ to a larger
subspace of $\Omega$ by $0$ off $V$ will also be denoted by $\mu$.
Thus $p=G_V\mu$ extends to the fine potential $G_{U\cap r(V)}\mu$
on $U\cap r(V)$ and further extends to a finely continuous function
$f:U\longrightarrow[0,+\infty]$ which equals $0$ off $U\cap r(V)$.

Suppose to begin with that $\mu$ is finite and (after extension to $\Omega$)
carried by a Euclidean compact subset $K$ of $\Omega$ contained in $r(V)$.
Let $q:={\widehat R}{}_f$ (sweeping relative to $U$).
Then $f\le G_U\mu$ and hence $q\le G_U\mu<+\infty$ q.e.\ on
$U$, and so $q$ is a fine potential on $U$ along with $G_U\mu$. On the other
hand, $q\le u$ because $f\le u$ on $U$. In the first place, $f=p\le u$ on $V$
and hence $f\le u$ on $U\cap r(V)\subset U\cap\widetilde V$ by fine continuity
of $f$ and $u$. Secondly, $f=0$ on $U\setminus r(V)$.
Since $p\preccurlyeq u$ on $V$ we have $u=p+s$ on $V$ for a certain
$s\in\cal S(V)$. Since $U\cap r(V)\setminus V$ is polar, $s$ extends by fine
continuity to a similarly denoted $s\in\cal S(U\cap r(V))$, and we have
$u=G_{U\cap r(V)}\mu+s$ on $U\cap r(V)$. Since $\fine\lim G_{U\cap r(V)}\mu=f=0$
at $U\cap\partial_fr(V)\subset U\setminus r(V)$ we have
$\fine\lim s=u$ at $U\cap\partial_fr(V)$. It follows by \cite[Lemma 10.1]{F1}
that the extension of $s$ to $U$ by $u$ on $U\setminus r(V)$ is of class
$\cal S(U)$. Denoting also this extension by $s$ we have $u=s+f$ on $U$.
By Mokobodzki's inequality in our setting, see \cite[Lemma 11.14]{F1},
we infer that $q={\widehat R}_f\preccurlyeq u$. Since $u$ is invariant on $U$
and $q\in\cal P(U)$ it follows that $q=0$ and hence $p\le q=0$ on $V$,
showing that indeed $u_{|V}$ is invariant.

Dropping the above temporary hypothesis that $\mu$ be finite and carried by
a Euclidean compact subset of $r(V)$ we decompose $\mu$ in accordance with
\cite[Lemma 2.3]{F4} into the sum of a sequence of finite measures $\mu_j$
with Euclidean compact supports $K_j\subset r(V)$.
Since $G_V\mu_j\preccurlyeq G_V\mu=p\preccurlyeq u_{|V}$, the
result of the above paragraph applies with $\mu$ replaced by $\mu_j$.
It follows that $G_V\mu_j=0$ and hence $p=G_V\mu=\sum_jG_V\mu_j=0$. Thus
${\widehat R}{}_u^A$ is indeed invariant on $V$.

(b) For each index $j$ there is by \cite[Theorem 4.4]{F4} a countable
finely open cover $(V_{jk})_k$ of $U_j$ such that
$\widetilde V_{jk}\subset r(U_j)=U_j$ and (with sweeping relative to $U_j$)
$${\widehat R}{}_{u|U_j}^{U_j\setminus V_{jk}}=u|U_j$$
for each $k$. It follows that (with sweeping relative to $U$, resp.\ $U_j$)
$$u\ge{\widehat R}{}_u^{U\setminus V_{jk}}\ge{\widehat R}{}_{u|U_j}^{U_j\setminus V_{jk}}=u$$
on each $U_j$, so equality prevails here. In particular,
$u={\widehat R}{}_u^{U\setminus V_{jk}}$ on $V_{jk}$ for each $j,k$. Consequently,
$u$ is invariant according to the quoted
theorem applied to the countable cover $(V_{jk})_{jk}$ of $U$.

(c) For indices $\alpha,\beta$ with $\alpha<\beta$
we have $u_\beta\le u_\alpha$ and hence $u_\beta\preccurlyeq u_\alpha$
by Lemma \ref{lemma2.1}.
The claim therefore follows from \cite[c), p.\ 132]{F1}.
\end{proof}

Throughout the rest of the article, $U$ is supposed (in the absence of other
indication) to be a {\it{regular}} fine domain in the Greenian domain
$\Omega\subset\RR^n$, $n\ge2$. In particular, $U$ is a Euclidean $K_\sigma$
subset of $\Omega$. We proceed to introduce and study the natural topology on
the $H$-cone $\cal S(U)$ of non-negative finely superharmonic functions on $U$.

\begin{theorem}\label{thm2.6}
There exists a resolvent family $(W_\lambda)$ of kernels on $U$ which are
absolutely continuous with respect to a measure $\sigma$ on $U$ such that
${\cal S}(U)$ is the cone of excessive functions which are finite
$\sigma$-a.e.
\end{theorem}

\begin{proof}
Let $p=G_\Omega\tau$ be a strict bounded continuous potentiel on the Greenian domain
$\Omega$ in $\RR^n$. Then the measure $\tau$ does not charge the polar sets
and we have $\tau(\omega)>0$ for any fine open subset of $\Omega$.
Denote by $V$ the Borel measurable kernel
on $\Omega$ defined by
$$Vf(x)=\int G_\Omega(x,y)f(y)d\tau(y)$$
for any Borel measurable function $f\ge 0$ in $\Omega$
and for $x\in \Omega$, and by $(V_\lambda)$ the resolvent family of kernels
whose kernel potential is $V$. According to \cite[Proposition 10.2.2, p. 248]{CC}, the cone of
of excessive functions of the resolvent $(V_\lambda)$ is the cone ${\cal S}(\Omega)\cup\{+\infty\}$
(and hence ${\cal S}(\Omega)$ is the cone cone of excessive functions of $(V_\lambda)$ which
are finite $\tau$-a.s).
Define a kernel $W$ on $U$ by
$$Wf=V{\bar f}-{\widehat R}{}_{V{\bar f}}^{\complement U}$$ (restricted to $U$)
for any Borel measurable function $f\ge 0$ in $U$,
where ${\bar f}$ denotes the extension of $f$ to $\Omega$ by $0$ in
$\Omega\setminus U$. Then  by \cite[Theorem 2.5]{BB} there exists a unique
resolvent family $(W_\lambda)$ of Borel measurable kernels having the potential
kernel $W$.

We proceed to determine the excessive functions for the resolvent
$(W_\lambda)$. Every superharmonic function $s\ge0$ on ${\cal S}(\Omega)$
is excessive for the resolvent $V_\lambda$,  and hence there exists by
\cite[Th\'eor\`eme 17, p.\ 11]{DM} an increasing sequence $(f_j)$ of
bounded Borel measurable functions $\ge 0$ such that $s=\sup Vf_j$.
It follows that $s-{\widehat R}{}_s^{\complement U}=\sup_j W(g_j)$, where $g_j$ denotes the restriction of $f_j$ to $U$. This shows that $s-{\widehat R}{}_s^{\complement U}$ is excessive for $(W_\lambda)$. For any $u\in\cal S(U)\cup\{+\infty\}$ there exists by \cite[Theorem 3]{F3a} a sequence of superharmonic functions $s_j\ge 0$
on $\Omega$ such that
$$u=\sup_j(s_j-{\widehat R}{}_{s_j}^{\complement U}),$$
where the sequence $(s_j-{\widehat R}{}_{s_j}^{\complement U})$ is increasing and hence $u$ is excessive for $(W_\lambda)$. Conversely, let $u$ be excessive
for $(W_\lambda)$. According to \cite[Th\'eor\`eme 17, p.\ 11]{DM} there exists
an increasing sequence of potentials $(Wf_j)$ of bounded Borel measurable functions $\ge0$ such that $u=\sup_jWf_j$. For each $j$ we have
$W(f_j)=V({\bar f}_j)-{\widehat R}{}_{V({\bar f}_j)}^{\complement U}$. But
 $V(\bar f_j)$ is finite and continuous on $U$, and so
${\widehat R}{}_{V(\bar f_j)}^{\complement U}$ is finely harmonic on $U$. It follows
that $W(f_j)\in {\cal S}(U)$. Consequently, $u$ is finely hyperharmonic on $U$, that is,
$u \in {\cal S}(U)\cup\{+\infty\}$.

Let $\sigma$ be the restriction of the measure $\tau$ to $U$. Then for any
$A\in \cal B(U)$ (the finely Borel $\sigma$-algebra on $U$)  such that $\sigma(A)=0$
we have $W1_A=(V\overline{1_A})_{|U}=0$, hence the resolvent  $(W_\lambda)$ is absolutely
continuous with respect to $\sigma$. Since $\tau$ does not charge the polar sets, we
see that $\cal S(U)$ is the cone of excessive functions which are finite $\sigma$-a.e.
This completes the proof of Theorem \ref{thm2.6}.
\end{proof}
	
It follows from Theorem \ref{thm2.6} by  \cite[Theorem 4.4.6]{BBC} that
${\cal S}(U)$ is a standard $H$-cone of fonctions on $U$.
Following \cite[Section 4.3]{BBC} we give ${\cal S}(U)$ the natural topology.
This topology on ${\cal S}(U)$ is metrizable and induced by the weak topology
on a locally convex topological vector space in which ${\cal S}(U)$ is embedded
as a proper convex cone. This cone is well-capped with compact caps, but we show that the cone ${\cal S}(U)$ even has a compact base, and that is crucial for our investigation. We shall need the following results from\cite{El1}:

\begin{theorem}\label{thm2.7} \cite[Lemme 3.5]{El1}. There exists a sequence
$(K_j)$ of  Euclidean compact subsets of \,$\Omega$
contained in $U$ and a polar set $P\subset U$ such that

{\rm{1.}} $U=P\cup\bigcup_j K'_j$, where $K'_j$ denotes the fine interior of
\,$K_j$.

{\rm{2.}} For any $j$ the restriction of any function from ${\cal S(U)}$ to
$K_j$ is l.s.c.\ in the Euclidean topology.
\end{theorem}

\begin{cor}\label{cor2.8} There exists a sequence $(H_j)$ of Euclidean compact
subsets of \,$U$, each non-thin at any of its points, and a polar set $P$
such that

{\rm{1.}} $U=P\cup\bigcup_j H_j$.

{\rm{2.}} For any $j$ the restriction of any function from
${\cal S}(U)$ to $H_j$ is l.s.c.\ in the Euclidean topology.
\end{cor}

\begin{proof} Write $U=P\cup\bigcup_j K'_j$ as in Theorem \ref{thm2.7}. Recall
that the fine interior of a subset of $\Omega$ is regular.  For each $j$ let
$(U_j^m)$ denote the fine components of $K'_j$; they are likewise regular.
For each couple
$(j,m)$ let $y_{j,m}$ be a point of $U_j^m$. And for each integer $n>0$ put $H_{j,m,n}=\{x\in U_j^m: G_{U_j^m}(x,y_{j,m})\ge \frac{1}{n}\}$. The sets $H_{j,m,n}$ are
Euclidean compact and non-thin at any of its points (in view of
\cite[Theorem 12.6]{F1}), and we have $U_j^m=\bigcup_nH_{j,m,n}$.
The sequence $(H_{j,m,n})$ and the polar set $P$ have the stated properties.
\end{proof}

\begin{remark}\label{remark2.10}
 The existence of a sequence $(K_j)$ of compact subsets of $U$ and a polar
set $P$ with $U=P\cup\bigcup_jK_j$ such that a given finely
continuous function (in particulier a superharmonic function) is continuous
relative to each $K_j$ follows from the pioneering work of Le Jan
\cite{LeJ1,LeJ2, LeJ3}, which applies more generally to the excessive functions
of the resolvant associated with the Hunt process. The weaker form of 1.\ in
our Theorem \ref{thm2.7} in which $U=P\cup\bigcup K_j'$ is replaced by the
condition $U=P\cup\bigcup K_j$, is a consequence of \cite[Corollaire 1.6]{BeB}
together with the existence of a family of universally continuous elements
which is increasingly dense in $\cal S(U)$.
In the present case our Theorem
\ref{thm2.7} is stronger than that of Beznea and Boboc.
In fact, our result is not a consequence of that of Beznea and Boboc because
 for a nest $(K_j)$ of $U$ the set $\bigcup_jK_j\setminus \bigcup_jK'_j$ is not
necessarily polar, as it is seen by the following example:

\smallskip
\textit{Example.} Let $A$ be a compact non-polar subset of $\Omega$ with empty
fine interior (for example $A$ can be a compact ball in some hyperplane in
$\RR^n$ such that $A\subset\Omega$). Let $\Omega_1=\Omega\setminus A$. Then $\Omega_1$ is
open and there exists an increasing sequence $(B_j)$ of 
open subsets of $\Omega_1$ such that $\overline B_j\subset \Omega_1$
for every $j$ ($\overline B_j$ denoting the Euclidean closure of $B_j$)
and that $\bigcup_jB_j=\Omega_1$. For any $j$ write $K_j=\overline B_j\cup A$.
Clearly, $(K_j)$ is an increasing sequence of compact subsets of $\Omega$ with
$\bigcup_jK_j=\bigcup_j\overline B_j\cup A=\Omega_1\cup A=\Omega$. It suffices to show
that $K'_j\subset\overline B_j$ for every $j$, for then
$\bigcup_jK'_j\subset\bigcup_j\overline B_j=\Omega_1=\Omega\setminus A$ with $A$
non-polar. Let $x\in K'_j$. If $x\in A$ then $V:=\Omega\setminus\overline B_j$ is
an open neighborhood of $x$ and $V\cap\overline B_j=\varnothing$. On the other hand,
$W:=K'_j$ is a fine neighborhood of $x$ contained in $K_j$.
Then $W\cap V\subset K_j$ and $(W\cap V)\cap\overline B_j=\varnothing$,
hence $x\in W\cap V\subset A$. But $W\cap V$ is finely open and $A'=\varnothing$,
so actually $x\notin A$, and since $x\in K'_j\subset K_j=\overline B_j\cup A$ we have $x\in\overline B_j$. Because this holds for every
$x\in K'_j$ we indeed have $K'_j\subset\overline B_j$.
\end{remark}

\begin{remark}
One may recover Corollary \ref{cor2.8} from \cite[Corollaire 1.6]{BeB}.
In fact, let $(K_j)$ be a sequence of compact subsets of $U$ and let $A$ be a
polar set with $U=A\cup\bigcup_jK_j$ such that the restriction of any
function $u\in {\cal S}(U)$ to each $K_j$ is l.s.c. One may suppose that all
the compact sets $K_j$ are non-polar. By repeated application of Ancona's
theorem \cite{An} it follows that each $K_j$ is the union of a polar set
$A_j$ and sequence $(K_{j,k})_k$ of compact sets $K_{j,k}$, each of which is
non-thin at each of its points. The double sequence $(K_{j,k})_{j,k}$,
arranged as a single sequence $(H_l)$, together with the polar set
$P:=A\cup \bigcup_jA_j$, meet the requirements in Corollary \ref{cor2.8}.
\end{remark}

We shall now use the sequence $(H_j)$ from this corollary to define in
analogy with \cite{Mo} a locally compact topology on the cone
${\cal S}(U)$. For each $j$ let ${\cal C}_l(H_j)$ denote the space of l.s.c.\ functions on $H_j$
with values in ${\overline \RR}{}_+$, and provide this space with the
topology of convergence in graph (cf.\ \cite{Mo}). It is known that
${\cal C}_l(H_j)$ is a compact metrizable space in this topology.
Let $d_j$ denote a distance compatible with this topology. We define
a pseudo-distance $d$ on ${\cal S}(U)\cup \{+\infty \}$ by
$$d(u,v)=\sum_j\frac{1}{2^j\delta({\cal C}_l(H_j))}d_j(u_{|H_j},v_{|H_j})$$
for each couple $(u,v)$ of functions from ${\cal S}(U)\cup \{+\infty \}$,
where $\delta({\cal C}_l(H_j))$ denotes the diameter of ${\cal C}_l(H_j)$.
Since two finely hyperharmonic functions are identical if the coincide quasi-everywhere it follows that $d$  is a true distance on ${\cal S}(U)\cup\{+\infty\}$. We denote by ${\cal T}$ the topology on ${\cal S}(U)\cup \{+\infty\}$ defined by  the distance $d$.

For any filter ${\cal F}$ on ${\cal S}(U)\cup\{+\infty\}$ we write
$$\underset{\cal F}{\lim{\widehat{\inf}}}
=\sup_{M\in {\cal F}}\underset{u\in M}{\widehat\inf}\,u,$$
where the l.s.c.\ regularized ${\widehat\inf}_{u\in M}u$ is taken with respect to the fine topology.

\begin{theorem}\label{thm2.9}\cite[Th\'eor\`eme 3.6]{El1}.
 The cone ${\cal S}(U)\cup \{+\infty\}$ is compact in the topology
${\cal T}$. For any convergent filter ${\cal F}$ on ${\cal S}(U)\cup\{+\infty\}$ we have
 $$\lim_{\cal F}=\underset{\cal F}{\lim\widehat{\inf}}.$$
\end{theorem}

\begin{proof}
Let ${\cal U}$ be an ultrafilter on ${\cal S}(U)\cup\{+\infty\}$.
For any $M\in {\cal U}$ put $u_M=\inf_{u\in M}u$. For each $j$ the ultrafilter
base ${\cal U}_j$ obtained from $\cal U$ by taking restrictions to the
Euclidean compact $H_j$ from Corollary \ref{cor2.8}, converges in the compact
space ${\cal C}_l(H_j)$ to the function $u_j:=\sup_{M\in {\cal U}}{\widehat {u_M}}^j$
(where ${\widehat v}^j$ for $v\in{\cal S}(U)\cup\{+\infty\}$ denotes the
finely l.s.c.\ regularized of the restriction of $v$ to $H_j$).
The finely l.s.c.\ regularized ${\widehat{u_M}}$ of $u_M$ in $U$
is l.s.c.\ in $H_j$ by Theorem \ref{thm2.7} and  minorizes $u_M$, whence ${\widehat {u_M}}
\le {\widehat {u_M}}^j$. On the other hand there exists a polar set
$A\subset\Omega$ such that $u_M={\widehat{u_M}}$
in $U\setminus A$, and so ${\widehat {u_M}}^j\le {\widehat {u_M}}$
in $H_j\setminus A$. But for $x\in A$ we have
${\widehat {u_M}}^j(x)\le {\widehat {u_M}}(x)$
because ${\widehat {u_M}}$ is finely continuous on $U$ and $x$ is in the fine
closure of $H_j\setminus A$ since $H_j$ is non-thin at $x$. We conclude that
$u_j=\lim{\widehat{\inf}}_{\cal U}$ in $H_j$
for each $j$. Since the function
$u:=\lim{\widehat{\inf}}_{\cal U}$ belongs to ${\cal S}(U)\cup\{+\infty\}$
according to \cite[\S12.9]{F1} it follows that the filter ${\cal U}$ converges
to $u$ in the topology ${\cal T}$. This proves that
${\cal S}(U)\cup \{+\infty\}$ is compact in the topology ${\cal T}$.
\end{proof}

\begin{cor}\label{cor2.10}
The topology of convergence in graph
coincides with the natural topology on ${\cal S}(U)$.
\end{cor}

\begin{proof}
This follows immediately from Theorem \ref{thm2.9}
and \cite[Theorem 4.5.8]{BBC}.
\end{proof}

\begin{cor}\label{cor2.11}
 The cone ${\cal S}(U)$ endowed with the natural topology has a compact base.
\end{cor}

\begin{proof} From Theorem \ref{thm2.9} and Corollary \ref{cor2.10} it follows
that the natural topology
on ${\cal S}(U)$ is locally compact, and we infer by a theorem of Klee
\cite[Theorem II.2.6]{Al} that indeed ${\cal S}(U)$ has a compact base.
\end{proof}

\begin{cor}\label{cor2.12} For given $x\in U$ the affine forms $u\mapsto u(x)$
and $u\mapsto {\widehat R}{}_u^A(x)$ ($A\subset U$) are l.s.c.\ in the natural
topology on ${\cal S}(U)$.
\end{cor}

\begin{proof}
Clearly, the map $u\mapsto \widehat R_u^A(x)$ is affine for fixed $x\in U$.
Let $(u_j)$ be a sequence in  ${\cal S}(U)$ converging naturally
to $u\in {\cal S}(U)$. For any index $k$ we have
$$\underset{j\ge k}{\widehat{\inf}}\,\widehat{R}{}_{u_j}^A(x)
\ge\widehat{R}{}_{{\widehat{\inf}}_{j\ge k}u_j}^A(x).$$
Either member of this inequality increases with $k$, and we get for
$k\to\infty$
$$\underset{k}{\lim\inf}\widehat{R}{}_{u_k}^A(x)
\ge\underset{k}{\lim\widehat{\inf}}\,\widehat{R}{}_{u_k}^A(x)
\ge\widehat{R}{}_{\lim\widehat{\inf}_k\,u_k}^A(x)=\widehat{R}{}_u^A(x),$$
and so the map
$u\mapsto\widehat{R}{}_u^A(x)$ is indeed l.s.c.\ on $\cal S(U)$.  For the map
$u\mapsto u(x)$ take $A=U$.
\end{proof}

In the rest of the present section we denote by $B$ a fixed compact base of
${\cal S}(U)$. As shown by Choquet (cf.\ \cite[Corollary I.4.4]{Al}) the set
 $\Ext(B)$ of extreme elements of
$B$ is a $G_\delta$ subset of $B$. On the other hand it follows by the fine Riesz decomposition theorem that every element of
$\Ext(B)$ is either a fine potential or an invariant function.
We denote by $\Ext_p(B)$ (resp.\ $\Ext_i(B)$) the cone of all
extreme fine potentials (resp.\ all extreme invariant functions) in $B$.
According to the  theorem  on integral representation of fine potentials \cite{F5}, any element of $\Ext_p(B)$ has the form $\alpha G_U(.,y)$,
where $\alpha$ is  a constant $>0$ and $y\in  U$.

\begin{prop}\cite[Proposition 4.3]{El1}\label{prop2.8}. $\Ext_p(B)$ and $\Ext_i(B)$ are Borel subsets of \,$B$.
\end{prop}

When $\mu$ is a non-zero Radon measure on $B$ there exists a unique element $s$ of $B$ such that
$$l(s)=\int_Bl(u)d\mu(u)$$
for every continuous affine form $l:B\longrightarrow[0,+\infty[\,$ on $B$
(in other words, $s$ is the barycenter of the probability measure
$\frac{1}{\mu(B)}\mu$). For any l.s.c.\ affine form
$\varphi:B\longrightarrow[0,+\infty]$ there exists by
\cite[Corollary I.1.4]{Al} an increasing sequence of continuous affine forms
on $B$  which converges to $\varphi$, and hence
$$\varphi(s)=\int_B\varphi(u)d\mu(u).$$
In particular, for fixed $x\in U$ and $A\subset U$, the affine forms
$u\mapsto u(x)$ and $u\mapsto {\widehat R}{}_u^A(x)$ are l.s.c.\ according to
Corollary \ref{cor2.12}, and hence
$$s(x)=\int_B u(x)d\mu(u)\text{ \;and \;}
{\widehat R}{}_s^A(x)=\int_B{\widehat R}{}_u^A(x)d\mu(u).$$

The following theorem is a particular case of Choquet's theorem.

\begin{theorem}\label{thm2.14}\cite[Th\'eor\`eme 4.1]{El1}
For any $s\in {\cal S}(U)$ there exists a unique Radon measure
on $B$ carried by $\Ext(B)$ such that
$$s(x)=\int_Bu(x)d\mu(u), \quad x\in U.$$
\end{theorem}

The next two theorems are immediate consequences of
\cite [Th\'eor\`eme 4.5]{El1}.

\begin{theorem}\label{thm2.15}
For any fine potential $p$ on $U$ there exists a unique Radon measure
$\mu$ on $B$ carried by $\Ext_p(B)$ such that
$$p(x)=\int_Bq(x)d\mu(q), \quad x\in U.$$
\end{theorem}

\begin{theorem}\cite[Th\'eor\`eme 4.6]{El1}\label{thm2.16}
For any invariant function $h\in {\cal S}(U)$ there exists a unique
Radon measure $\mu$ on $B$ carried by $\Ext_i(B)$ such that
$$h(x)=\int_Bk(x)d\mu(k), \quad x\in U.$$
\end{theorem}

The following theorem is an immediate consequence of Theorems \ref{thm2.15} and
\ref{thm2.16}.

\begin{theorem}\label{thm2.17}
Let $A\subset \Ext_p(B)$ (resp.\ $A\subset \Ext_i(B)$), and let
$\mu$ be a Radon measure on $B$. Then $\int_Aud\mu(u)$ is a fine potential
(resp.\ an invariant function).
\end{theorem}

\begin{remark}\label{remark2.18} In view of \cite[Section 11.16]{F1} the set
$\cal H_i(U)$ of all invariant functions on $U$ is clearly a lower complete and conditionally upper complete sublattice of $\cal S(U)$ in the specific order.
According to Lemma \ref{lemma2.1}  the specific order on $\cal H_i(U)$ coincides
with the pointwise order.
\end{remark}

\begin{remark}\label{remark2.14} On a Euclidean domain the invariant functions
are the same as the harmonic functions $\ge0$. It is well known in view of
Harnack's principle that the set of these functions is closed in
$\cal H(U)$ with the natural topology (which coincides with the topology of
R.-M.\ Herv\'e \cite{He}). However, when $U$ is just a regular fine domain in a
Green space $\Omega$, the set $\cal H_i(U)$ of invariant functions on $U$
need not be closed in the induced natural topology on $U$, as shown by the
following example: Let $y\in U$ be a Euclidean non-inner point of $U$,
and let $(y_k)$ be a sequence of points of $\Omega\setminus U$ which converges
Euclidean to $y$. The sequence $(G_\Omega(.,y_k)_{|U})$ then converges naturally
in $\cal S(U)$ to $G_\Omega(.,y)_{|U}$, which does not belong to $\cal H_i(U)$
because its fine potential part $G_U(.,y)$ is non-zero.
\end{remark}

\section{Martin compactification of $U$ and integral representation
in $\cal S(U)$}\label{sec3}

We continue considering a regular fine domain $U$ in a Greenian domain
$\Omega\subset\RR^n$, $n\ge2$.
Let $B$ be a compact base of the cone ${\cal S}(U)$ and let
$\Phi:\cal S(U)\longrightarrow[0,+\infty[$ be a continuous affine form such that
$$B=\{u\in {\cal S}(U): \Phi(u)=1\}.$$
Then $\Phi(u)>0$ except at $u=0$.
Consider the mapping $\varphi:U\longrightarrow B$ defined by
$$\varphi(y)=P_y=\frac{G_U(.,y)}{\Phi(G_U(.,y))}.$$
Note that $\phi$ is injective because
$y=\{x\in U:G_U(x,y)=+\infty\}=\{x\in U:\phi(y)(x)=+\infty\}$. We may therefore
identify $y\in U$ with $\phi(y)=P_y\in B$ and hence $U$ with $\varphi(U).$

We denote by ${\overline U}$ the closure of $U$ in $B$ (with the natural
topology), and write $\Delta(U)={\overline U}\setminus U$.
Then ${\overline U}$ is compact in $B$ and is called the Martin
compactification of $U$, and $\Delta(U)$ is called the
Martin boundary of $U$.

If $B$ and $B'$ are two bases of ${\cal S}(U)$ the Martin compactifications of
$U$ relative to $B$ and to $B'$ are clearly homeomorphic.

Throughout the rest of this article we fix the compact base
$B$ of the cone ${\cal S}(U)$ and the above continuous affine form
$\Phi:{\cal S}(U)\longrightarrow\,]0,+\infty[\,$ defining this base.

For any $Y\in \overline U$ consider the function
$K(.,Y)\in B\subset{\cal S}(U)\setminus\{0\}$ defined on $U$ by
$K(x,Y)=\phi(Y)(x)$ if $Y\in U$ and $K(.,Y)=Y$ if $Y\in\Delta(U).$ Clearly
the map $Y\longmapsto K(.,Y)$ is a bijection of $\overline U$ on $B$.

\begin{definition}\label{def3.1} The function
$K:U\times\overline{U}\longrightarrow\,]0,+\infty]$ defined by
$K(x,Y)=K(.,Y)(x)$
is called the (fine) Riesz-Martin kernel for $U$, and its restriction
to $U\times\Delta(U)$ is called the (fine) Martin kernel for $U$.
\end{definition}

\begin{prop}\label{prop3.1b} The Riesz-Martin kernel
$K:U\times\overline U\longrightarrow\,]0,+\infty]$
has the following  properties, $\overline U$ being given the natural topology:

{\rm{(i)}} For any $x\in U$, $K(x,.)$ is l.s.c.\ on $\overline U$.

{\rm{(ii)}} For any $Y\in \overline U$, $K(.,Y)\in{\cal S}(U)$ is finely
continuous on $U$.

{\rm{(iii)}} $K$ is l.s.c.\ on $U\times \overline U$ when $U$ is given the fine
topology and $\overline U$ the natural topology.
\end{prop}

\begin{proof} (i) follows from Corollary \ref{cor2.12} applied to $u=K(.,Y)$
while identifying $K(.,Y)$ with $Y$.

(ii) is obvious.

(iii) Let $x_0\in U$, $Z\in\overline U$, and let $(V_j)$ be a fundamental
system of open neighborhoods of $Z$ in ${\overline U}$ such that
$V_{j+1}\subset V_j$ for any $j$. For a given constant $c>0$ consider
the increasing sequence of functions
$$k_j:={\inf}_{Y\in V_j}K(.,Y)\wedge c$$
and their finely l.s.c.\ regularizations ${\widehat k_j}\in{\cal S}(U)$.
By the Brelot property, cf.\ \cite{F3}, there exists a fine
neighborhood $H$ of  $x_0$ in $U$ such that $H$ is compact (in the Euclidean
topology) and that the restrictions of the
functions ${\widehat k}_j\in {\cal S}(U)$ and of $K(.,Z)\wedge c\in\cal S(U)$
to $H$ are continuous on $H$ (again with the induced Euclidean topology).
By (i) we have pointwise on $U$
$$K(.,Z)\wedge c=\underset{Y\to Z}{\liminf}\,K(.,Y)\wedge c
=\sup_j\underset{Y\in V_j}{\inf}\,K(.,Y)\wedge c,$$
which quasi-everywhere and hence everywhere on $U$ equals
$\sup_j\underset{Y\in V_j}{\widehat\inf}\,K(.,Y)\wedge c\in\cal S(U)$.
By Theorem \ref{thm2.9} and Dini's theorem there exists for given $\eps>0$
an integer $j_0>0$ such that
$$K(.,Z)\wedge c=\sup_j\underset{Y\in V_j}{\widehat{\inf}}K(.,Y)\wedge c
=\sup_j{\widehat k}{}_j<{\widehat k}{}_i+\eps$$
on $H$ for any $i\ge j_0$. For any fine
neighborhood $W$ of $x_0$ with $W\subset H$ we have
\begin{eqnarray*}
\inf_{x\in W,\,Y\in V_j}K(x,Y)\wedge c
&=&\inf_{x\in W}\,k_j(x)\ge\inf_{x\in W}\,{\widehat k}{}_j(x)\\
&\ge&\inf_{x\in W}\,K(x,Z)\wedge c-\eps\ge K(x_0,Z)\wedge c-2\eps
\end{eqnarray*}
for $j\ge j_0$. The assertion (iii) follows by letting $\eps\to0$ and next
$c\to+\infty$.
\end{proof}

\begin{remark}\label{remark3.1c} The Riesz-Martin kernel $K$ is in general not
l.s.c.\ in the product of the induced natural topology on $U$ and the natural
topology on $\overline U$. Not even the function
$K(.,y)=G_U(.,y)/\Phi(G_U(.,y))$, or equivalently $G_U(.,y)$ itself, is
l.s.c.\ on $U$ with the induced natural topology for fixed $y\in U$. For if
the set $V:=\{x\in U:G_U(x,y)>1\}$ were open for every $y\in U$ then $U$ would
be a natural neighborhood of $y$ for every $y\in U$, that is,
$U$ would be naturally open in $\overline U$. But that is in general not the
case because $\overline U\setminus U=\Delta(U)$ need not be naturally compact,
see Example \ref{example3.5} below.
\end{remark}

\begin{remark}\label{remark3.1d} A set $A\subset U$ is termed a Euclidean
nearly Borel set if it differs by a polar set from a Euclidean Borel set.
We denote by $\cal B(U)$, resp.\ $\cal B^*(U)$, the $\sigma$-algebra of all
Euclidean Borel, resp.\ nearly Borel subsets of $U$. Every finely open set
$V\subset U$ is Euclidean nearly Borel because its regularized $r(V)$ is a
Euclidean $F_\sigma$-set and $r(V)\setminus V$ is polar. It follows that every
open subset $W$ of $U\times\overline U$, now with the fine topology on $U$
(and of course the natural topology on $\overline U$), belongs to the
$\sigma$-algebra $\cal B^*(U)\times\cal B(\overline U)$ generated by all sets
$A_1\times A_2$ where $A_1\in\cal B^*(U)$ and where
$A_2\in\cal B(\overline U)$, that is, $A_2$ is a Borel subset of
$\overline U$. In view of Proposition \ref{prop3.1b} (iii) every set
$\{(x,Y)\in U\times\overline U:K(x,Y)>\alpha\}$ ($\alpha\in\RR$) is such an
open set $W$ and therefore belongs to $\cal B^*(U)\times\cal B(\overline U)$.
This means that the Riesz-Martin kernel $K$ is measurable
with respect to $\cal B^*(U)\times\cal B(\overline U)$.
\end{remark}

\begin{definition}\label{def3.2}
An invariant function $h\in {\cal S}(U)$ is termed \textit{minimal}
if it belongs to an extreme generator of the cone
${\cal S}(U).$
\end{definition}

Recall that $\Ext(B)$ denotes the set of extreme points of $B$, and
$\Ext_i(B)$ the subset of minimal invariant functions in $B$.

\begin{prop}\label{prop3.3}
For any point $Y\in \Delta(U)$, if the function $K(.,Y)$ is an
extreme point of $B$, then  $K(.,Y)$ is a minimal invariant function.
\end{prop}

\begin{proof}
By the Riesz decomposition of functions from ${\cal S}(U)$, $K(.,Y)$
is either a minimal fine potential on $U$ or else a minimal
invariant function on $U$.  In the former case it follows by the
integral representation of fine potentials that $K(.,Y)$ must be equal
to $P_y$ for some $y\in U$, which contradicts $Y\in\Delta(U)$. Thus
$K(.,Y)$ is indeed a minimal invariant function. The converse is obvious.
\end{proof}

\begin{definition}\label{def3.4}
A  point $Y\in \Delta(U)$ is termed minimal if the function $K(.,Y)$ is minimal,
that is, it belongs to an extreme generator of the cone ${\cal S}(U).$
\end{definition}

We denote by $\Delta_1(U)$ the set of all minimal points of $\Delta(U)$.
Contrary to the case where $U$ is Euclidean open in $\Omega$,
$\Delta(U)$ is in general not compact (in the natural topology), as
shown by the following example.

\begin{example}\label{example3.5} Let $\omega$ be a H\"older domain in
$\RR^n$ ($n\ge 2)$ such that $\omega$ is irregular with a single
irregular boundary point $z$ (for example a Lebesgue spine), and take
$U=\omega\cup\{z\}$. According to \cite[Theorems 1 and 3.1]{Ai} the
Euclidean boundary $\partial \omega$ of $\omega$ is contained in
$\Delta(\omega)$. It follows that $z$ belongs to the Euclidean closure
of $\Delta(\omega)\setminus \{z\}$. But
$\Delta(U)=\Delta(\omega)\setminus\{z\}$, where $z$ is identified with
$P_z$, and since $\Delta(\omega)$ is compact we infer that
$\Delta(U)$ is noncompact.
\end{example}

However, we have the following
\begin{prop}\label{prop3.6}
$\Delta(U)$ is a $G_\delta$ of \,$\overline U$.
\end{prop}

\begin{proof}
Let $(B_l)$ be a sequence of open balls in $\RR^n$ with Euclidean closures
$\overline B_l\subset\Omega$ and such that $\Omega=\bigcup_lB_l$.
For integers $k,l>0$ put
$$A_{kl}= \{y\in U: \Phi(G_U(.,y))\ge 1/k\}\cap {\overline B_l}.$$
The sets $A_{kk}$ cover $U$ because $G_U(.,y)>0$ and hence
$\Phi(G_U(.,y))>0$. We first show that each $A_{kl}$
is compact in $\overline U$ with the natural topology. Let $(y_j)$ be a
sequence of points of $A_{kl}$.
After passing to a subsequence we may suppose that $(y_j)$
converges Euclidean to a point $y\in {\overline B_l}$, and that the
sequences $(G_U(.,y_j))$ and $({\widehat R}{}_{G_{\Omega}(.,y_j)}^{\complement U})$
(restricted to $U$) converge in ${\cal S}(U)\cup\{+\infty\}$. It follows that
(with $G_\Omega(.,y_j)$ restricted to $U$)
$$\Phi(G_\Omega(.,y_j))
=\Phi(G_U(.,y_j))+\Phi({\widehat R}{}_{G_{\Omega}(.,y_j)}^{\complement U}),$$
and hence by passing to the limit in $\cal S(U)\cup\{+\infty\}$
as $j\to+\infty$
\begin{eqnarray}
\Phi(G_\Omega(.,y))
=\lim_j\Phi(G_U(.,y_j))+\lim_j\Phi({\widehat R}{}_{G_{\Omega}(.,y_j)}^{\complement U}).
\end{eqnarray}
On the other hand,
$${\widehat R}^{\complement U}_{G_\Omega(.,y)}
={\widehat R}^{\complement U}_{\lim{\widehat{\inf}_j}G_\Omega(.,y_j)}\le
\underset{j}{\lim{\widehat{\inf}}}\,{\widehat R}^{\complement U}_{G_\Omega(.,y_j)}.$$
The restriction of the function ${\widehat R}^{\complement U}_{G_\Omega(.,y)}$
to $U$ being invariant according to Lemma \ref{lemma2.3} it follows by Lemma
\ref{lemma2.1} that
${\widehat R}{}^{\complement U}_{G_\Omega(.,y)}\preccurlyeq
  {\lim\widehat{\inf}}_j{\widehat R}^{\complement U}_{G_\Omega(.,y_j)}$
(after restriction to  $U$ here, and often in the rest of the proof),
and hence
$$\Phi({\widehat R}{}^{\complement U}_{G_{\Omega}(.,y)})\le
\Phi(\lim\widehat{\inf}{}_j\widehat R{}^{\complement U}_{G_\Omega(.,y_j)})
=\lim_j\Phi({\widehat R}^{\complement U}_{G_\Omega(.,y_j)}).$$
We infer from (3.1) by the definition of $A_{kl}$ that
$$\Phi(G_\Omega(.,y))\ge \frac{1}{k}
+\Phi({\widehat R}^{\complement U}_{G_\Omega(.,y)}),$$
and consequently $y\in U$ and $\Phi(G_U(.,y))\ge 1/k$, whence
$y\in A_{kl}$.
Now put $s={\lim}_jG_U(.,y_j).$
We have $s>0$ because $y_j\in A_{kl}\subset U$ and hence
$$\Phi(s)=\lim_j\Phi(G_U(.,y_j))\ge1/k.$$
Hence
$$\phi(y_j)
=\frac{G_U(.,y_j)}{\Phi(G_U(.,y_j))}\to\frac{s}{\Phi(s)}\in\cal S(U).$$
This shows that $A_{kl}$ is compact in
the natural topology on $\overline U$, and consequently
$\Delta(U)=\overline U\setminus U= \bigcap_{kl}(\overline U\setminus A_{kl})$
is indeed a $G_\delta$ in $\overline U$.
\end{proof}

\begin{cor}\label{cor3.7a}
$\Delta_1(U)$ is a $G_\delta$ of ${\overline U}$.
\end{cor}

\begin{proof}
Because $\Delta_1(U)=\Delta(U)\cap \Ext_i(B)$ it suffices according to
Proposition \ref{prop3.6}  to use the well-known fact that $\Ext(B)$ is a
$G_\delta$ of $B$.
\end{proof}

The proof of Proposition \ref{prop3.6} also establishes

\begin{cor}\label{cor3.8} For any integers $k,l>0$ the set
$C_{kl}:=\{P_y: y\in A_{kl}\}$
is compact in $B$, and the mapping $\varphi: y\mapsto P_y$
is a homeomorphism of \,$A_{kl}$ onto $C_{kl}$.
\end{cor}

\begin{cor}\label{cor3.9}
Let $K\subset U$ be compact in the Euclidean topology. Then the set
$C_K:=\{P_y: y\in K\}$ is a (natural) $K_\sigma$ in $B$, and so is therefore
$\Ext_p(B)=\{P_y:y\in U\}$.
\end{cor}

\begin{proof}
Since $C_K=\bigcup_k\varphi(K\cap A_k)$
the assertion follows by the preceding corollary.
\end{proof}

\begin{prop}\label{prop3.10}
Let $u\in {\cal S}(U)$ and let $A\subset{\overline A}\subset U$, where
${\overline A}$ denotes the closure of \,$A$ in
${\overline U}$. The measure $\mu$ on $B$
carried by the extreme points of \,$B$
and representing ${\widehat R}^A_u$
is then carried by ${\overline A}$.
\end{prop}

\begin{proof}
Let $p$ be a finite fine potential $>0$ on $U$. For any pair $(k,l)$ and any
integer $j>0$ the function ${\widehat R}{}_{u\wedge jp}^A$
is a fine potential on $U$ and finely harmonic on $U\setminus {\overline A}$
by Lemma \ref{lemma2.3}. The measure $\mu_j$ on $B$ carried by $\Ext(B)$
and representing ${\widehat R}{}_{u\wedge jp}^A$ is carried by ${\overline A}$.
The sequence of probability measures
$\frac{1}{\Phi(\widehat R{}_{u\wedge jp}^A)}\mu_j$ has a subsequence
$(\mu_{j_k})$ which converges to a probability measure $\mu$ on
$B$ carried by ${\overline A}$. We thus have
$${\widehat R}{}_u^A=\lim_{k\to\infty}{\widehat R}{}_{u\wedge j_kp}^A
=\Phi({\widehat R}{}_u^A)\lim_{k\to\infty}\int q\,d\mu_{j_k}(q)
=\Phi({\widehat R}{}_u^A)\int q\,d\mu(q).$$
The assertion now follows by the fact that
${\overline A}\subset U\subset \Ext(B)$ and from the uniqueness of the
integral representation in Choquet's theorem.
\end{proof}

Recall from the beginning of the present section the continuous affine form
$\Phi\ge0$ on $\cal S(U)$
such that the chosen compact base $B$ of the cone $\cal S(U)$ consists of
all $u\in\cal S(U)$ with $\Phi(u)=1$. Also consider the compact sets
$A_{kl}\subset U$ in the proof of Proposition \ref{prop3.6}. Cover $\Omega$
by a sequence of Euclidean open balls $B_k$ with closures
$\overline B_k$ contained in $\Omega$.

\begin{lemma}\label{lemma4.2}
{\rm{(a)}} The mapping $U\ni y\longmapsto G_U(.,y)\in\cal S(U)$
is continuous from $U$ with the fine topology into $\cal S(U)$ with the
natural topology.

{\rm{(b)}} The function $U\ni y\longmapsto\Phi(G_U(.,y))\in \,]0,+\infty[$ is
finely continuous.

{\rm{(c)}} The sets
$$V_k=\{y\in U:\Phi(G_U(.,y))>1/k\}\cap B_k$$
form a countable cover of \,$U$ by finely open sets which are
relatively naturally compact in $U$.
\end{lemma}

\begin{proof} (a) Consider a net $(y_\alpha)_{\alpha\in I}$ on $U$ converging finely
to some $y\in U$, in other words $s(y_\alpha)\to s(y)$ for every
$s\in\cal S(U)$. Taking $s=G_U(x,.)$ with $x\in U$ we have in particular
$G_U(x,y_\alpha)\to G_U(x,y)$ for every $x\in U$. Since
$G_U(.,y_\alpha)\in\cal S(U)\cup\{+\infty\}$, which is naturally compact by
Theorem \ref{thm2.9}, we may assume that $G_U(.,y_\alpha)\to z$ naturally for
some $z\in\cal S(U)\cup\{+\infty\}$. Furthermore, we then have
$$z=\underset{\alpha}{\lim\widehat{\inf}}\,G_U(.,y_\alpha)=G_U(.,y),$$
where the latter equality holds first q.e.,\ and next everywhere on $U$ by fine
continuity of $z$ and
$G_U(.,y)$. Thus $G_U(.,y_\alpha)\to G_U(.,y)$, which proves assertion (a).

(b) It follows from (a) that this function is finely continuous on $U$ because
the mapping $\Phi:\cal S(U)\to[0,+\infty\,[$
is (naturally) continuous, as recalled above.

(c) According to (b), each $V_k$ is finely open along with $U\cap B_k$.
By the proof of Proposition \ref{prop3.6} we have
$V_k\subset A_{kk}\subset U$ with $A_{kk}$ naturally compact.
Clearly, $\bigcup_kV_k=U$.
\end{proof}

\begin{cor}\label{cor4.3} The mapping
$\phi:U\ni y\longmapsto P_y=K(.,y)\in\cal S(U)$
is continuous from $U$ with the fine topology to
$\cal S(U)\cup\{+\infty\}$ with the natural topology.
In other words, the fine topology on $U$ is finer than the topology on $U$
induced by the natural topology on $\overline U$.
\end{cor}

\begin{proof} Recall that $P_y=G_U(.,y)/\Phi(G_U(.,y))$ for $y\in U$.
For any net $(y_\alpha)$ on $U$ converging finely to $y\in U$ we obtain from
(a) and (b) in the above lemma
$$\underset{\alpha}{\lim\widehat{\inf}}\,P_{y_\alpha}
=\frac{\lim\inf_\alpha G_U(.,y_\alpha)}{\Phi(G_U(.,y))}
=\frac{G_U(.,y))}{\Phi(G_U(.,y))}=P_y,
$$
first quasi-everywhere, and next everywhere in $U$ by fine continuity.
\end{proof}

\begin{theorem}\label{thm3.12}
Every extreme element of the base $B$ of the cone
${\cal S}(U)$ belongs to ${\overline U}$. In particular, any extreme invariant
function $h$ in $B$ has the form $h=K(.,Y)$ where $Y\in\Delta_1(U)$.
\end{theorem}

\begin{proof}
Let $p$ be an extreme element of $B$. By  Riesz decomposition, either
$p$ is the fine potential of a measure supported by a single point $y\in U$, or else $p$ is an invariant function on $U$.
In the former case we have $p=P_y$ and hence $p\in U$. In the latter case it follows by the proof of Proposition
\ref{prop3.6} that there exists an increasing sequence of compact  subsets
$(K_j)$ of ${\overline U}$ (of the form $A_{kl}$) such that
$\bigcup_jK_j=U$. For each $j$, ${\widehat R}{}_p^{K_j}$ is a fine potential
on $U$. In fact, by Proposition \ref{prop3.10}, the measure $\mu$
on $B$ carried by $\Ext(B)$ and representing ${\widehat R}{}_p^{K_j}$ is carried
by $K_j$, and hence
$${\widehat R}{}_p^{K_j}=\int P_yd\mu(y)=\int G_U(.,y)d\nu(y),
$$
is a fine potential on $U$, the measure $\nu$ on $B$ being well defined by
$d\nu(y)=\frac1{\Phi(G_U(.,y))}d\mu(y)$, cf.\ Corollary \ref{cor4.3} (b).
For any $j$ there exists a Radon measure $\mu_j$ on $B$ such that
${\widehat R}{}_p^{K_j}=\int q\,d\mu_j(q)$. The measure $\mu_j$ is carried by
${\overline U}$. Because $p$ is invariant it follows by Lemma \ref{lemma2.1}
that the sequence ${\widehat R}{}_p^{K_j}$ increases specifically to $p$ as
$j\to\infty$, and the sequence $(\mu_j)$ is therefore increasing. Consequently,
$\int d\mu_j=\Phi({\widehat R}{}_p^{K_j})\to \Phi(p)$
as $j\to\infty$, and the sequence $(\mu_j)$ converges vaguely to a measure
$\mu$ on $\overline U$. It follows that
$p=\int_{\overline U}q\,d\mu(q)$, and since $p$ is extreme we conclude that
$p\in\overline U$ because $\mu$ must be carried by a single point.
\end{proof}

\begin{cor}\label{cor3.13}
$\Ext(B)=U\cup\Delta_1(U)$.
\end{cor}

\begin{proof}
Every extreme element of $B$ is either the fine potential of a measure
supported by a single point $y\in U$, hence of the form
$P_y$, or else a minimal invariant functions, hence of the form
$K(.,Y)$ with $Y\in \Delta_1(U)$, according to Proposition \ref{prop3.3}.
This establishes the inclusion $\Ext(B)\subset U\cup\Delta_1(U)$.
The opposite inclusion is evident.
\end{proof}

\begin{theorem}\label{thm3.14}
For any invariant function $h$ on $U$ there exists a unique Radon measure
$\mu$ on ${\overline U}$ carried by $\Delta_1(U)$ such that
$$h(x)=\int_{\Delta(U)} K(x,Y)d\mu(Y), \quad x\in U.$$
\end{theorem}

\begin{proof}
The theorem follows immediately from Theorem \ref{thm2.15}  and Corollaries
\ref{cor3.8} and \ref{cor3.13}.
\end{proof}

\begin{prop}\label{prop3.15} Let $u\in\cal S(U)$ and let $V$ be
a finely open Borel subset of $U$. Let $\mu$ be the measure on $B$
carried by $\Ext(B)$ and representing $u$. Then $\mu(V)=0$ if and only if
the restriction $u_{|V}$ is invariant.
\end{prop}

\begin{proof} Write $u=p+h$ with $p\in\cal P(U)$ and $h$ invariant on $U$.
Let $\lambda$ and $\nu$ be the measures on $\Ext(B)$ representing $p$ and
$h$, respectively. Then $\mu=\lambda+\nu$ with $\nu(U)=0$ according to
Corollary \ref{cor3.13} or \ref{thm2.16} and Theorem \ref{thm3.14}.
Writing $\Phi(G_U\mu)=\alpha$ we have by \cite[Lemma 2.6]{F4}
$$\alpha p=G_U\lambda=G_V\lambda_{|V}+\widehat R{}_{G_U\lambda}^{U\setminus V}
\preccurlyeq\alpha u\quad\text{\rm{on }}V,$$
and hence $G_V\lambda_{|V}\preccurlyeq\alpha u$ on $V$. If $u_{|V}$ is
invariant the fine potential $G_V\lambda_{|V}$ must therefore be $0$,
whence $\lambda(V)=0$ and finally $\mu(V)=\lambda(V)+\nu(V)=0$ because
$\nu(V)\le\mu(V)=0$. Conversely, suppose that $\mu(V)=0$ and hence
$\lambda(V)=0$. In the above display
$\widehat R{}_{G_U\lambda}^{U\setminus V}$ is invariant on $V$ by Lemma
\ref{lemma2.3}, and so is therefore $p$. It follows that $u=p+h$ is
invariant on $V$, $h$ being invariant on $U$ and hence on $V$ by Theorem
\ref{thm2.5} (a).
\end{proof}

For any (positive) Borel measure $\mu$ on $\overline U$ define a function
$K\mu:U\longmapsto[0,+\infty]$ by
$$K\mu=\int K(.,Y)d\mu(Y),\quad x\in U.$$
This integral exists in view of Proposition \ref{prop3.1b} (i).

\begin{theorem}\label{thm3.21} {\rm{1.}} For any Borel measure $\mu$ on
$\overline U$, $K\mu$ is finely hyperharmonic on $U$, that
is, $K\mu\in{\cal S}(U)\cup\{+\infty\}$.

{\rm{2.}} Every function $u\in{\cal S}(U)$ has a unique integral
representation $u=K\mu$ in terms of a Borel measure $\mu$ on
$U\cup\Delta_1(U)$.
\end{theorem}

\begin{proof} 1. The kernel $K$ is measurable with respect to the product
$\sigma$-algebra $\cal B^*(U)\times\cal B(\overline U)$ in view of the
conclusion in Remark \ref{remark3.1d}. It follows that the function $K\mu$ is
nearly Euclidean Borel measurable on $U$.  We begin by showing
that $K\mu$ is nearly finely hyperharmonic, cf.\ \cite[Definition 11.1]{F1}.
Let $V\subset U$ be finely open with $\widetilde V\subset U$. For any $x\in V$
the swept measure $\eps_x^{\Omega\setminus V}$ is carried by the fine boundary
$\partial_fV\subset U$ and does not charge any polar set. Hence
$\eps_x^{\Omega\setminus V}$ may be regarded as a measure on the $\sigma$-algebra
$\cal B^*(U)$ of all Euclidean nearly Borel subsets of $U$, cf.\ again Remark
\ref{remark3.1d}. Altogether, Fubini's theorem applies, and we obtain
for any $x\in U$:
\begin{eqnarray*} \widehat{R}{}_{K\mu}^{U\setminus V}(x)
&=&\int_UK\mu(y)d\eps_x^{\Omega\setminus V}(y)\\
&=&\int_{\overline U}\biggl(\int_UK(y,Y)d\eps_x^{\Omega\setminus V}(y)\biggr)d\mu(Y)\\
&\le&\int_{\overline U}K(x,Y)d\mu(Y)=K\mu(x),
\end{eqnarray*}
the inequality because $K(.,Y)\in\cal S(U)$, cf.\ \cite[Definition 8.1]{F1}.
Thus $K\mu$ is nearly finely hyperharmonic on $U$. To show that $K\mu$ is
actually finely hyperharmonic we shall prove that the finely l.s.c.\ envelope
$\widehat{K\mu}$ of $K\mu$ equals $K\mu$, cf.\ \cite[Lemma 11.2 and
Definition 8.4]{F1} according to which
$$\widehat{K\mu}(x)=\sup_{V\in\cal V}\,\int_UK\mu\,d\eps_x^{\Omega\setminus V},$$
where $\cal V$ denotes the lower directed family of all finely open sets
$V\subset U$ of Euclidean compact closure in $\Omega$ contained in $U$.
For each $V\in\cal V$ and $x\in V$ we have, again by Fubini in view of Remark
\ref{remark3.1d},
\begin{eqnarray*} \int_UK\mu(y)d\eps_x^{\Omega\setminus V}(y)
&=&\int_{\overline U}\biggl(\int_UK(y,Y)d\eps_x^{\Omega\setminus V}(y)\biggr)d\mu(Y)\\
&=&\int_{\overline U}\widehat{R}{}_{K(.,Y)}^{U\setminus V}(x)d\mu(Y).
\end{eqnarray*}
Taking supremum over all $V\in\cal V$ leads to $\widehat{K\mu}=K\mu$ as desired.
In fact, the increasing net of finely superharmonic functions
$\bigl(\widehat{R}{}_{K(.,Y)}^{U\setminus V}\bigr)_{V\in\cal V}$ admits an increasing
subsequence
with the same pointwise supremum, by \cite[Remark (p.\ 91)]{F1}, and this
supremum equals $\widehat{R}_{K(.,Y)}^{U\setminus\{x\}}=\widehat{R}_{K(.,Y)}^U=K(.,Y)$.
It follows that
\begin{eqnarray*} \int_UK\mu\,d\eps_x^{U\setminus V}
&=&\int_{\overline U}\widehat{R}_{K(.,Y)}^{U\setminus V}(x)d\mu(Y)\\
&\nearrow&\int_{\overline U}K(x,Y)d\mu(Y)=K\mu(x)
\end{eqnarray*}
according to \cite[Theorem 11.12]{F1}. We conclude that $\widehat{K\mu}=K\mu$,
and so $K\mu$ is indeed finely hyperharmonic on $U$.

2. As noted in \cite[Theorem 4.1]{El1} this follows immediately from
Choquet's integral representation theorem applied to the cone ${\cal S}(U)$ with the base $B$.
\end{proof}

\begin{lemma}\label{lemma3.21a} For any set $A\subset U$ and
any Radon measure $\mu$ on $\overline U$ we have
$\widehat R{}_{K\mu}^A=\int\widehat R{}_{K(.,Y)}^Ad\mu(Y)$.
\end{lemma}

\begin{proof} As in the proof of Lemma \ref{lemma2.2},
${\widehat R}{}_{K\mu}^A$ and ${\widehat R}{}_{K(.,Y)}^A$
remain unchanged when $A$ is replaced by $b(A)\cap U$, and we may therefore
assume that $A=b(A)\cap U$, whence ${\widehat R}{}_{K\mu}^A=R_{K\mu}^A$ and
${\widehat R}{}_{K(.,Y)}^A=R_{K(.,Y)}^A$.
As in the proof of Theorem \ref{thm3.21} the kernel $K$ on
$U\times\overline U$ is measurable with respect to
$\cal B^*(U)\times\cal B(\overline U)$. Furthermore, $K\mu$ is finely
hyperharmonic on $U$, and we have by Lemma \ref{lemma2.2} and Fubini
for any $x\in U$
\begin{eqnarray*}
\widehat R{}_{K\mu}^A(x)
&=& R_{K\mu}^A(x)=\int_U K\mu\,d\eps_x^{A\cup(\Omega\setminus U)}\\
&=& \int_{\overline U}d\mu(Y)\int_UK(.,Y)d\eps_x^{A\cup(\Omega\setminus U)}
=\int_{\overline U}R_{K(.,Y)}^A(x)d\mu(Y)\\
&=& \int_{\overline U}\widehat R{}_{K(.,Y)}^A(x)d\mu(Y).
\end{eqnarray*}
\end{proof}

\begin{cor}\label{cor3.18}
Let $\mu$ be a Borel measure on $\Delta_1(U)$ and let $h=K\mu$.
If \,$h$ is finite q.e.\ then $h\in\cal S(U)$ and $h$ is invariant.
\end{cor}

\begin{proof} Let $V$ be a regular finely open set such that
$\widetilde V\subset U$.
For any $x\in V$ the measure $\eps_x^{\Omega\setminus V}$ is carried by
$\partial_fV\subset\widetilde V\subset U$ and does not
charge any polar set. This measure may therefore be regarded as a measure on
the $\sigma$-algebra $\cal B^*(U)$ of all nearly Borel subsets of $U$, cf.\
Remark \ref{remark3.1d}, where it is also shown that $K$ is measurable with
respect to the product $\sigma$-algebra $\cal B^*(U)\times\cal B(\overline U)$.
Supposing that $h<+\infty$ q.e.\ on $U$ we may therefore
apply Fubini's theorem and Lemma \ref{lemma2.2} to obtain
\begin{eqnarray*}
{\widehat R}{}_h^{U\setminus V}(x) &=& R_h^{U\setminus V}(x)
=\int_Uh(y)d\eps_x^{\Omega\setminus V}(y)\\
&=& \int_{\Delta_1(U)}\biggl(\int_UK(y,Y)d\eps_x^{\Omega\setminus V}(y)\biggr)d\mu(Y)\\
&=& \int_{\Delta_1(U)}R_{K(.,Y)}^{U\setminus V}(x)d\mu(Y)
=\int_{\Delta_1(U)}{\widehat R}{}_{K(.,Y)}^{U\setminus V}(x)d\mu(Y)\\
&=& \int_{\Delta_1(U)}K(x,Y)d\mu(Y)=h(x).
\end{eqnarray*}
In the remaining case $x\in U\setminus V$ the resulting equation
${\widehat R}{}_h^{U\setminus V}(x)=h(x)$ holds because $x$ belongs to
$U\setminus V$ which is a base relative to $U$. According to
\cite[Theorem 4.4]{F4} there is a countable cover of $U$ by sets like the
above set $V$, and since ${\widehat R}{}_h^{U\setminus V}(x)=h(x)$ for each such
set, we conclude from the quoted theorem that indeed $h$ is invariant.
\end{proof}

\begin{remark}\label{remark3.19}
The finiteness condition on $h$ in Corollary \ref{cor3.18} is equivalent with
$h\ne +\infty$, that is $h\in {\cal S}(U).$
\end{remark}

\begin{cor}\label{cor3.20}
For any finite measure $\mu$ on $\Delta_1(U)$ the function
$h=\int_{\Delta_1(U)}K(.,Y)d\mu(Y)$
is an invariant function.
\end{cor}

\begin{proof}
Let $\nu$ be the measure on $B$ defined by $\nu(A)=\mu(\Delta_1(U)\cap A)$
for any Borel set $A\subset B$. Then $\nu$ is a finite mesure on $B$, and we may suppose that $|\nu|=1$. Let $h\in B$ be the barycenter of $\nu$.
Then $k=\int_{\Delta_1(U)}K(.,Y)d\mu(Y)=h$, and hence $h$ is
invariant according to Theorem \ref{thm2.15} and Corollary \ref{cor3.13}.
\end{proof}

\begin{cor}\label{cor3.22} Let $\mu$ be a Borel measure on
$U\cup\Delta_1(U)$. A function $u=K\mu\not\equiv+\infty$ is
a fine potential, resp.\ an invariant function, if and only $\mu$ is
carried by $U$, resp.\ by $\Delta_1(U)$.
\end{cor}

\begin{proof} For $y\in U$ write $\alpha(y):=\Phi(G_U(.,y))$.
If $\mu$ is carried by $U$ then $K\mu(y)=
G_U({\alpha(y)}^{-1}\mu)$ is a fine potential on $U$. Conversely, if $u$ is
a fine potential on $U$ then there is a measure $\nu$ on $U$ such
that $u(y)=G_U\nu(y)= K(\alpha(y)\nu)$. On the other hand, if $\mu$ is
carried by $\Delta_1(U)$ then $u=p+h$ with $p=K\mu$ for some $\mu$
on $U$ and $h=K\lambda$ for some $\lambda$ carried by $\Delta_1(U)$. By
uniqueness in Theorem \ref{thm3.21}, $\mu=\nu+\lambda$, where $\nu$ is
carried by $U$ and hence $\nu=0$, so that $K\mu=K\lambda$ is invariant
according  to Corollary \ref{cor3.18}. Conversely, if $K\mu$ is invariant
then $u=p+h=K\nu+K\lambda$ as above, and here $h$ is invariant according to
Corollary \ref{cor3.20}. It follows that $p=0$, that is $\nu=0$, and so
$\mu=\lambda$ is carried by $\Delta_1(U)$.
\end{proof}

\begin{cor}\label{cor3.11}
Let $u\in {\cal S}(U)$ and $A\subset {\overline A}\subset U$, where
${\overline A}$
denotes the closure of \,$A$ in
${\overline U}$. Then
${\widehat R}^A_u$
is a fine potential on $U$.
\end{cor}

\begin{proof} According to Proposition \ref{prop3.10} the measure $\mu$ on
$U\cup\Delta_1(U)$ representing ${\widehat R}^A_u$ is carried by
$\overline A$. It follows by the preceding corollary that
${\widehat R}^A_u$ is a fine potential.
\end{proof}

We close this section with the following characterizations of the
invariant functions and the fine potentials on $U$. These characterizations
are  analogous (but only partly comparable) to
\cite[Theorems 4.4 and 4.5]{F4}, respectively, where the
present condition $\overline V\subset U$ was replaced by the weaker condition
$\widetilde V\subset U$.

\begin{theorem}\label{thm3.16}
Let $u\in {\cal S}(U)$. Then $u$ is invariant if and only if
${\widehat R}{}_u^{U\setminus V}=u$
for any regular finely open set $V\subset U$ such that the closure
${\overline V}$ of \,$V$ in the natural topology on $B$ is contained in $U$.
\end{theorem}

\begin{proof}
Suppose that $u$ is invariant. For any $V$ as stated we have
$u\le {\widehat R}{}_u^V+{\widehat R}{}_u^{U\setminus V}$.
By the Riesz decomposition property \cite[p.\ 129]{F1} there are functions
$u_1,u_2\in {\cal S}(U)$ such that
$u=u_1+u_2$ with
$u_1\le {\widehat R}{}_u^V$ and $u_2 \le{\widehat R}{}_u^{U\setminus V}$.
But ${\widehat R}{}_u^V$ is a fine potential by Corollary \ref{cor3.11}, and
so is therefore $u_1$. It follows that $u_1=0$ because $u_1\preccurlyeq u$
and $u$ is invariant. Consequently, $u\le {\widehat R}{}_u^{U\setminus V}$,
and so indeed ${\widehat R}{}_u^{U\setminus V}=u$.
Conversely, suppose that ${\widehat R}{}_u^{U\setminus V}=u$
for any regular finely open (and finely connected, if we like) set
$V\subset U$ with ${\overline V}\subset U$. Let $\mu$ be the (finite) measure on
$U\cup\Delta_1(U)$ which represents $u$. Then
$$u=\int_UP_yd\mu(y)+\int_{\Delta_1(U)}K(.,Y)d\mu(Y).$$
For any regular fine domain $V\subset U$ with
${\overline V}\subset U$ we have by hypothesis and by Lemma \ref{lemma3.21a}
$$u={\widehat R}{}_u^{U \setminus V}
=\int_U {\widehat R}^{U\setminus V}_{P_y}d\mu(y)+
\int_{\Delta_1(U)}{\widehat R}^{U\setminus V}_{K(.,Y)}d\mu(Y),$$
and hence
$$\int_U(P_y-{\widehat R}^{U\setminus V}_{P_y})d\mu(y)=0.$$
Since $V$ is regular and since
$P_y-\widehat R{}_{P_y}^{U\setminus V}=G_V(.,y)/\Phi(G_U(.,y))>0$ on $V$ in view of
\cite[Lemma 2.6]{F4}, it follows
that $\mu(V)=0$. This implies by Lemma \ref{lemma4.2} (c) (together with the
fact that a finely open set has only countably many fine components) that $\mu$
is carried by $\Delta_1(U)$. Consequently, $u$ is invariant according to Theorem
\ref{thm2.17}.
\end{proof}

\begin{cor}\label{3.17}
Let $u\in {\cal S}(U)$. Then $u$ is a fine potential if and only if
${\widehat{\inf_j}}{\widehat R}{}_u^{U\setminus V_j}=0$
for some, and hence any, cover of \,$U$ by an increasing sequence $(V_j)$
of regular finely open sets
such that ${\overline V_j}\subset V_{j+1}\subset U$ for every $j$.
\end{cor}

\begin{proof}
Suppose first that $u$ is a fine potential, and consider any cover $(V_j)$ of
$U$ as stated. Denote $v={\widehat{\inf}}_j\widehat R{}_u^{U\setminus V_j}$,
which is likewise a fine potential.
For each index $k$, $\widehat R{}_u^{U\setminus V_j}$ is invariant on $V_k$ for any
$j\ge k$ according to Lemma \ref{lemma2.3}. It follows by Theorem \ref{thm2.5}
(c) that $v$ likewise is invariant on $V_k$. By varying $k$ we see from
Theorem \ref{thm2.5} (b) applied to the regular finely open sets $V_k$ that
$v$ is invariant on $\bigcup_kV_k=U$. Consequently $v=0$.
Conversely, suppose that ${\widehat\inf}_j\widehat
R{}_u^{U\setminus V_j}=0$ for some cover $(V_j)$ as stated in the corollary.
We have $u=p+h$, where $p$ is a fine
potential and $h$ is an invariant function. By the preceding theorem,
$h={\widehat R}{}_h^{U\setminus V_j}\le {\widehat R}{}_u^{U\setminus V_j}$
for every $j$, and hence $h=0$, showing that $u$ is a fine potential.
\end{proof}

\section{The Fatou-Na{\"\i}m-Doob theorem for\\
finely superharmonic functions}\label{sec4}

As mentioned in the Introduction, this section is inspired by the axiomatic
approach to the Fatou-Na{\"\i}m-Doob theorem given in \cite{T}. These axioms
are, however, only partially fulfilled in our setting. In particular, our
invariant functions, which play the role of positive harmonic functions, may
take the value $+\infty$. We therefore choose to give the proof of the
Fatou-Na{\"\i}m-Doob theorem without reference to the proofs in \cite{T}.

We continue considering a regular fine domain $U$ in a Greenian domain
$\Omega\subset\RR^n$, $n\ge2$.  Recall that ${\cal P}(U)$ denotes the
band in ${\cal S}(U)$ consisting of all fine potentials on $U$, and that
the orthogonal band $\cal P(U)^\perp=\cal H_i(U)$ relative to ${\cal S}(U)$
consists of all invariant functions $h$ on $U$; these are characterized
within $\cal S(U)$ by their integral representation $h=K\mu$, that is,
 $$
h(x)=K\mu(x)=\int K(x,Y)d\mu(Y)
 $$
in terms of a unique measure $\mu$ on $\overline U$ carried by the minimal
Martin boundary $\Delta_1(U)$ (briefly: a measure on $\Delta_1(U)$), see
Theorem \ref{thm3.21} and Corollary \ref{cor3.22}. In the
present section we shall not consider the whole Riesz-Martin space
$\overline U$ and the full Riesz-Martin kernel
$K:U\times\overline U\longrightarrow\;]0,+\infty]$ (Definition \ref{def3.1}),
but only the Martin boundary $\Delta(U)$ and the Martin
kernel, the restriction of the Riesz-Martin kernel to $U\times\Delta(U)$, and
$K$ will henceforth stand for this restriction. It is understood that $U$
and $\Delta(U)$ are given the natural topology (induced by the natural
topology on the Riesz-Martin space $\overline U)$).
In particular, by Lemma \ref{lemma4.2}(c), $U$ is a
$K_\sigma$ subset of $\overline U$, and we know that $\Delta(U)$ and the
minimal Martin boundary $\Delta_1(U)$ are  $G_\delta$ subsets of
$\overline U$ (Proposition \ref{prop3.6} and Corollary \ref{cor3.7a}).
We shall need the following two preparatory lemmas.

\begin{lemma}\label{lemma7.2}
{\rm(a)} If \,$h_1,h_2\in\cal H_i(U)$, $p\in{\cal P}(U)$, and if
\,$h_1\le h_2+p$, then $h_1\preccurlyeq h_2$ and $h_2-h_1\in{\cal H}{}_i(U)$.

{\rm(b)} If \,$(h_j)$ is an increasing sequence of functions
$h_j\in\cal H_i(U)$ majorized by some $u\in{\cal S}(U)$ then
$\sup_j h_j\in\cal H_i(U)$.
\end{lemma}

It is understood in (a) (and similarly elsewhere) that $h_2-h_1$ is
defined to be the extension by fine continuity from $\{x\in
U:h_1(x)<+\infty\}$ to $U$, cf.\ \cite[Theorem 9.14]{F1}. (Equivalently,
$h_2-h_1$ is well defined because $\cal S(U)$ is an $H$-cone, as noted after
the proof of Theorem \ref{thm2.6}.) For any set
$A\subset U$, $R{}_u^A$ and ${\widehat R}{}_u^A$ are understood as reduction
and sweeping of a function $u$ on $U$ relative to $U$, whereas
$\eps_x^A$ stands for sweeping of $\eps_x$ on $A$ relative to
all of $\Omega$.

\begin{proof} (a) There exists $s\in{\cal S}(U)$ such
that $h_1+s=h_2+p$. We have $s=h+q$ with $h\in\cal H_i(U)$ and
$q\in{\cal P}(U)$.
It follows that $(h_1+h)+q=h_2+p$ with $h_1+h\in\cal H_i(U)$, and
hence $h_1+h=h_2$ (and $q=p$).

(b) According to (a) the sequence $(h_j)$ is even specifically
increasing. Because $\sup_ju_j\in{\cal S}(U)$ along with $u$ we have
$\sup_ju_j=\curlyvee_jh_j$ (\cite[(b), p.\ 132]{F1}, which belongs to the band
$\cal H_i(U)$ along with each $h_j$.
\end{proof}

\begin{lemma}\label{lemma7.3}
For any set $E\subset U$ and any $Y\in\Delta_1(U)$ we have
${\widehat R}{}_{K(.,Y)}^E\ne K(.,Y)$ if and only if
$\widehat R{}_{K(.,Y)}^E\in{\cal P}(U)$.
 \end{lemma}

\begin{proof} Note that, for any $u\in{\cal S}(U)$ and $E\subset U$, we have
 $R{}_u^E\ne u\;\Longleftrightarrow\;{\widehat R}{}_u^E\ne u$.
Proceeding much as in \cite[proof of G)]{T}, see also \cite{Go}, we
first suppose that $\widehat R{}_{K(.,Y)}^E\ne K(.,Y)$. Since
$R_{K(.,Y)}^E\le K(.,Y)$ there
exists $x_0\in U$ with $R_{K(.,Y)}^E(x_0)<K(x_0,Y)$. Thus there exists
$u\in{\cal S}(U)$ such that $u\ge K(.,Y)$ on $E$ and $u(x_0)<K(x_0,Y)$.
Replacing $u$ by $u\wedge K(.,Y)\in{\cal S}(U)$ we arrange that $u\le K(.,Y)$
on all of $U$. Writing $u=q+h$ with $q\in{\cal P}(U)$ and $h\in{\cal
U_i}(U)$ we have $h\le u\le K(.,Y)$ on $U$. Because $K(.,Y)\in\cal H_i$ is
minimal invariant and $h\preccurlyeq K(.,Y)$ by Lemma \ref{lemma7.2}
(with $p=0$), there exists $\alpha\in[0,1]$ such that
$h=\alpha K(.,Y)$, and here $\alpha<1$ since $h(x_0)\le
u(x_0)<K(x_0,Y))$. On $E$ we have $q=u-h=K(.,Y)-h=(1-\alpha)K(.,Y)$, and
hence
 $$
K(.,Y)=p:=\frac1{1-\alpha}q\quad\text{on }E.
 $$
 Thus $p\in{\cal P}(U)$ and $\widehat R{}_{K(.,Y)}^E=\widehat R{}_p^E\le p$ on
$U$, so indeed $\widehat R{}_{K(.,Y)}^E\in\cal P(U)$. Conversely, suppose that
$\widehat R{}_{K(.,Y)}^E=K(.,Y)$ and (by contradiction) that
$\widehat R{}_{K(.,Y)}^E\in\cal P(U)$. Being thus  both a fine potential and
invariant, $K(.,Y)$ must equal $0$, which is false.
\end{proof}

\begin{definition}\label{def7.4}  A set $E\subset U$ is said to be
minimal-thin at a point $Y\in\Delta_1(U)$ if \,$\widehat
R^E_{K(.,Y)}\ne K(.,Y)$, or equivalently (by the preceding lemma)
if \,$\widehat R{}_{K(.,Y)}^E$ is a fine potential on $U$.
\end{definition}

\begin{cor}\label{cor7.5} For any $Y\in\Delta_1(U)$ the sets $W\subset U$
for which $U\setminus W$ is minimal-thin at \,$Y$ form a filter on $U$.
 \end{cor}

This follows from Lemma \ref{lemma7.3} which easily implies that for any
$W_1,W_2\subset U$ such that ${\widehat R}{}_{K(.,Y)}^{U\setminus W_i}\ne K(.,Y)$
for $i=1,2$, we have ${\widehat R}{}_{K(.,Y)}^{U\setminus(W_1\cup W_2)}\ne K(.,Y)$.

The filter from Corollary \ref{cor7.5} is called the minimal-fine filter
at $Y$ and will be denoted by $\cal F(Y)$.
A limit along that filter will be called a minimal-fine limit and will be
denoted by $\lim_{\cal F(Y)}$.

For any two functions $u,v\in\cal S(U)$ with $v\ne0$ the quotient $u/v$
is assigned some arbitrary value, say $0$, on the polar set of points
at which both functions take the value $+\infty$. The choice of such a
value does not affect a possible minimal-fine limit of $u/v$  at a point
$Y\in\Delta_1(U)$ because every polar set $E$ clearly is minimal-thin
at any point of $\Delta_1(U)$.

For any two measures $\mu,\nu$ on a measurable space we denote by
$d\nu/d\mu$ the Radon-Nikod\'ym derivative of the absolutely continuous
component of $\nu$ relative to that of $\mu$, cf.\ e.g.\ \cite[p.\ 773]{Do},
\cite[p.\ 305f]{AG}.

For any function $u\in\cal S(U)$ with Riesz decomposition $u=p+h$, where
$p\in\cal P(U)$ and $h\in{\cal H}{}_i(U)$, we denote by $\mu_u$ the unique
measure on $\Delta_1(U)$ which represents the invariant part $h$ of $u$,
that is, $h=\int K(.,Y)d\mu_u(Y)$.

We may now formulate the Fatou-Na{\"\i}m-Doob theorem in the present
setting of finely superharmonic functions on a regular fine domain $U$.
It clearly contains the
classical Fatou-Na{\"\i}m-Doob theorem for which we refer to
\cite[1.XII.19]{Do}, \cite[9.4]{AG}.

\begin{theorem}\label{thm7.6} Let $u,v\in\cal S(U)$, where $v\ne0$. Then
$u/v$ has minimal-fine limit $d\mu_u/d\mu_v$ at $\mu_v$-a.e.\ point $Y$ of
$\Delta_1(U)$.
\end{theorem}

For the proof of Theorem \ref{thm7.6} we begin by establishing the following
important particular case, cf.\ \cite[Theorem 1.2]{T}.

\begin{prop}\label{prop7.7} {\rm{(\cite{T}.)}} Let $u\in\cal S(U)$ and
$h\in\cal H_i(U)\setminus\{0\}$, and suppose that $u\wedge h\in{\cal P}(U)$.
Then $u/h$ has minimal-fine limit $d\mu_u/d\mu_h=0$ at $\mu_h$-a.e.\ point $Y$
of \,$\Delta_1(U)$.
\end{prop}

\begin{proof} Write $u\wedge h=p$.  Given $\alpha\in\;]0,1[$, consider any
point $Y\in\Delta_1(U)$ such that $\lim_{\cal F(Y)}\frac{u}{h}>\alpha$. Then
$\{u\le\alpha h\}\notin\cal F(Y)$, that is, $\{u>\alpha h\}$ is not
minimal-thin at $Y$:
\begin{eqnarray}
{\widehat R}{}_{K(.,Y)}^{\{u>\alpha h\}}=K(.,Y).
\end{eqnarray}
It follows that (always with $Y$ ranging over $\Delta_1(U)$)
\begin{eqnarray}
\{Y:\lim\sup_{\cal F(Y)}\,\frac uh>\alpha\}\subset
A_\alpha:=\{Y:{\widehat R}_{K(.,Y)}^{\{u>\alpha h\}}=K(.,Y)\}.
\end{eqnarray}
We show that $\mu_h(A_\alpha)=0$. Consider the measure $\nu=1_{A_\alpha}\mu_h$ on
$\Delta_1(U)$ and the corresponding function
 $$
v=\int K(.,Y)d\nu(Y)
=\int{\widehat R}{}_{K(.,Y)}^{\{u>\alpha h\}}d\nu(Y)=\widehat
R{}_{K\nu}^{\{u>\alpha h\}}={\widehat R}{}_v^{\{u>\alpha h\}},
 $$
the second equality by (4.1)
and the third equality by Lemma \ref{lemma3.21a}. Since
$v\le\int K(.,Y)d\mu_h=h$ and $0<\alpha<1$ it follows that
 $$
v={\widehat R}{}_v^{\{u>\alpha h\}}\le{\widehat R}{}_h^{\{u>\alpha h\}}
\le\frac u\alpha\wedge h\le\frac u\alpha\wedge\frac h\alpha=\frac p\alpha.
$$
Because $v=\int K(.,Y)d\nu(Y)\in\cal H_i(U)$ and $p/\alpha\in{\cal P}(U)$
we find by Lemma
\ref{lemma7.2} (a) (applied to $h_1=v$, $h_2=0$) that $v=0$, that is,
 $$
v=\int K(.,Y)d\nu(Y)=\int_{A_\alpha}K(.,Y)d\mu_h(Y)=0,
 $$
and since $K(.,Y)>0$ it follows that $\mu_h(A_\alpha)=0$. By
varying $\alpha$ through a decreasing sequence tending to $0$ we
conclude from (4.2)
that indeed $\mu_h(\{Y:{\limsup}_{\cal F(Y)}\,u/h>0\})=0$.
\end{proof}

The rest of the proof of Theorem \ref{thm7.6} proceeds much as in the classical
case. For the convenience of the reader we bring  most of the details,
following in part \cite[Section 9.4]{AG}.

\begin{cor}\label{cor7.8} Let $u,h\in\cal H_i(U)$, where $h\ne0$ and
$\mu_u,\mu_h$ are mutually singular. Then $u\wedge h\in\cal P(U)$, and $u/h$
has minimal-fine limit $0$ $\mu_h$-a.e.\ on $\Delta_1(U)$.
\end{cor}

\begin{proof} There are Borel subsets $A_1,A_2$ of $U$ such that
$A_1\cup A_2=\Delta_1(U)$ and
$\mu_u(A_1)=\mu_h(A_2)=0$. Write $u\wedge h=p+k$ with $p\in\cal P(U)$,
$k\in{\cal H}{}_i(U)$. Then $k\le u$, hence $k\preccurlyeq u$ by Lemma
\ref{lemma2.1}, and so $\mu_k\le\mu_u$.
Similarly, $\mu_k\le\mu_h$. It follows that
$\mu_k(\Delta_1(U))\le\mu_u(A_1)+\mu_h(A_2)=0$ and hence $k=0$, and so
$u\wedge h=p\in\cal P(U)$. The remaining assertion now follows from
Proposition \ref{prop7.7}.
\end{proof}

\begin{cor}\label{cor7.10} Let $h\in\cal H_i(U)\setminus\{0\}$, let $A$ be a
Borel subset of \,$\Delta(U)$, and let
$h_A=K(1_A\mu_h)=\int_AK(.,Y)d\mu_h(Y)$. Then
$h_A/h$ has minimal-fine limit $1_A(Y)$ at $Y$ $\mu_h$-a.e.\ for
$Y\in\Delta_1(U)$.
\end{cor}

\begin{proof} Write $u=h-h_A\in\cal H_i(U)$, which is invariant because
$h_A\preccurlyeq h$. Since $h_A=K(1_A\mu_h)$ and because $1_A\mu_h$ is carried by
$\Delta_1(U)$ along with $\mu_h$, we have $\mu_{h_A}=1_A\mu_h$, which is
carried by $A\cap\Delta_1(U)$. Similarly,
$\mu_u=\mu_h-\mu_{h_A}=1_{U\setminus A}\mu_h$ is carried by
$\Delta_1(U)\setminus A$. In particular, $\mu$ and $\mu_h$ are mutually
singular. It follows by Corollary \ref{cor7.8} that $h_A/h$ and
$u/h=1-h_A/h$ have minimal fine limit $0$ $\mu_h$-a.e. on $A$ and on
$\Delta_1(U)\setminus A$, respectively, whence the assertion.
\end{proof}

\begin{definition}\label{def7.11} For any function
$h\in\cal H_i(U)\setminus\{0\}$
and any $\mu_h$-integrable function $f$ on $\Delta_1(U)$ we define
$$u_{f,h}(x)=\int K(x,Y)f(Y)d\mu_h(Y)\quad\text{for }x\in U.
$$
\end{definition}

\begin{prop}\label{prop7.12} Let $h\in\cal H_i(U)\setminus\{0\}$ and let $f$
be a $\mu_h$-integrable function on $\Delta(U)$. Then $u_{f,h}/h$ has
minimal-fine limit $f(Y)$ at $\mu_h$-a.e.\ point $Y$ of \,$\Delta_1(U)$.
\end{prop}

\begin{proof} We may assume that $f\ge0$. The case where $f$ is a (Borel)
step function follows easily from Corollary \ref{cor7.10}. For the general
case we refer to the proof of \cite[Theorem 9.4.5]{AG}, which carries over
entirely.
\end{proof}

We are now prepared to prove Theorem \ref{thm7.6}, the Fatou-Na{\"\i}m-Doob theorem in our setting.

\smallskip\noindent{\it{Proof of Theorem \ref{thm7.6}.}} Write $v=p+h$ with $p\in\cal P(U)$ and $h\in{\cal H}_i(U)$. By our definition of $\mu_v$ we then have $\mu_v=\mu_h$.
We may assume that $h\ne0$, for otherwise $\mu_v=0$ and the assertion becomes
trivial. Let $\nu$ be the singular component of $\mu_u$ with
respect to $\mu_v=\mu_h$. Write $u=q+k$ with $q\in\cal P(U)$ and
$k\in\cal H_i(U)$. Then $\mu_u=\mu_k=f\mu_h+\nu$, and hence in view of Definition \ref{def7.11}
$$u=q+u_{f,h}+\int K(.,Y)d\nu(Y).$$
By applying Proposition \ref{prop7.7} with $u$ replaced by $q$
(hence $\mu_u$ by $\mu_q=0$), next
Corollary \ref{cor7.8} with $u$ replaced by $K\nu$ (hence $\mu_h$ replaced by $\nu$ and $d\mu_u/d\mu_h$ by $d\nu/d\mu_h=0$), and finally by applying Proposition \ref{prop7.12} to the present $u_{f,h}$, we see that $u/h$ has minimal-fine limit
$f(Y)$ at $\mu_v$-a.e.\ point $Y$ of $\Delta_1(U)$. Since $u/v$ is defined
quasi-everywhere in $U$ and
$$\frac uv=\frac{u/h}{1+p/h},$$
the theorem now follows by applying Proposition \ref{prop7.7} with $u$ there replaced by $p$ (and hence $\mu_u$ by the present $\mu_p=0$). \hfill$\square$

\end{document}